\documentclass[reqno]{amsart}

\usepackage[english]{babel}
\usepackage{csquotes}
\usepackage[style=alphabetic,
            isbn=false,
            doi=false,
            url=false,
            backend=bibtex,
            maxnames=5, 
            maxalphanames=5, firstinits=true
           ]{biblatex}
\bibliography{bib_carea}
\usepackage{amsmath,amsfonts,amsthm,mathtools, amssymb}
\DeclareMathSymbol{\shortminus}{\mathbin}{AMSa}{"39}
\usepackage{xcolor}
\usepackage{tikz-cd}
\usepackage{hyperref}
\usepackage{aliascnt}
\usepackage{xparse}
\usepackage{caption}
\usepackage{ mathrsfs }
\usepackage{tabto}
\usepackage{imakeidx}
\usepackage[normalem]{ulem}
\usepackage{comment}
\allowdisplaybreaks

\RequirePackage{tikz}
\RequirePackage{tikz-cd}
\usetikzlibrary{babel, backgrounds}
\usetikzlibrary{matrix,arrows,calc}
\usetikzlibrary{positioning}
\usetikzlibrary{decorations.pathreplacing,decorations.markings,decorations.pathmorphing}
\usetikzlibrary{positioning,arrows,patterns}
\usetikzlibrary{cd}
\usetikzlibrary{intersections}
\usepackage{pgfplots}
\tikzset{
	noderect/.style={
		rectangle,
		draw,
		inner sep=0,
		minimum size=4pt
	},
	nodecirc/.style={
		circle,
		draw,
		inner sep=0,
		minimum size=4pt
	}
}

\newlength\leveldist
\newlength\hordist
 \RequirePackage{filecontents}

\indexsetup{level=\subsection*,toclevel=\paragraph,headers={ \MakeUppercase{Euler sequence on strata}}{\MakeUppercase{Euler sequence on strata}}, noclearpage,firstpagestyle=headings}
\makeindex[name=graph,title=Graphs and levels,columns=1,columnseprule,options=-s Idx.ist]

\usepackage{tikz}
\usetikzlibrary{matrix,arrows,calc}
\usetikzlibrary{positioning}
\usetikzlibrary{decorations.pathreplacing,decorations.markings,decorations.pathmorphing}
\usetikzlibrary{positioning,arrows,patterns}
\usetikzlibrary{cd}
\usetikzlibrary{intersections}

\def \PotlI{.85} % position of top left indices 
\def \PotrI{.85} % position of top right indices 

\tikzset{
	every loop/.style={very thick},
	comp/.style={circle,fill,black,inner sep=0pt,minimum size=5pt},
	order bottom left/.style={pos=.05,left,font=\tiny},
	order top left/.style={pos=.9,left,font=\tiny},
	order bottom right/.style={pos=.05,right,font=\tiny},
	order top right/.style={pos=.9,right,font=\tiny},
	order node dis/.style={text width=.75cm},
	circled number/.style={circle, draw, inner sep=0pt, minimum size=12pt},
	below left with distance/.style={below left,text height=10pt},
    below right with distance/.style={below right,text height=10pt}
	}

\makeatletter
\ifcsname phantomsection\endcsname
    \newcommand*{\@gobblenexttocentry}[9]{}
\else
    \newcommand*{\@gobblenexttocentry}[4]{}
\fi
\newcommand*{\addsubsection}{%
    \addtocontents{toc}{\protect\@gobblenexttocentry}%
    \subsection*}
\makeatother

\DeclareFieldFormat[article]{title}{{\it #1}}
\DeclareFieldFormat{journaltitle}{{\rm #1}}
\renewbibmacro{in:}{%
  \ifentrytype{article}{}{\printtext{\bibstring{in}\intitlepunct}}}
\DeclareFieldFormat[incollection]{title}{{\it #1}}
\DeclareFieldFormat{journaltitle}{{\rm #1}}

\begin{document}

\def\subsectionautorefname{Section}
\def\subsubsectionautorefname{Section}
\def\sectionautorefname{Section}
\def\equationautorefname~#1\null{(#1)\null}

%color TikZ
\definecolor{sqsqsq}{rgb}{0.12549019607843137,0.12549019607843137,0.12549019607843137}

\newcommand{\mynewtheorem}[4]{
  % #1=env name
  % #2=displayed name
  % #3=counter to share
  % #4=counter to obey
  \if\relax\detokenize{#3}\relax %#3 empty
    \if\relax\detokenize{#4}\relax %#4 empty
      \newtheorem{#1}{#2}
    \else
      \newtheorem{#1}{#2}[#4]
    \fi
  \else
    \newaliascnt{#1}{#3}
    \newtheorem{#1}[#1]{#2}
    \aliascntresetthe{#1}
  \fi
  \expandafter\def\csname #1autorefname\endcsname{#2}
}

%%%%%%%%%%%%%%%%%%%%%%%%%%%%
%%%% Benutzerteil: %%%%%%%%%%%%%%%%%
%%%%%%%%%%%%%%%%%%%%%%%%%%%%
\mynewtheorem{theorem}{Theorem}{}{section}
\mynewtheorem{lemma}{Lemma}{theorem}{}
\mynewtheorem{rem}{Remark}{lemma}{}
\mynewtheorem{prop}{Proposition}{lemma}{}
\mynewtheorem{cor}{Corollary}{lemma}{}
\mynewtheorem{add}{Addendum}{lemma}{}
\mynewtheorem{definition}{Definition}{lemma}{}
\mynewtheorem{question}{Question}{lemma}{}
\mynewtheorem{assumption}{Assumption}{lemma}{}
\mynewtheorem{example}{Example}{lemma}{}

%%%%%%%%%%%%%%%%%%%%%%%%%%%%
%%%%%%%%%%%%%%%%%%%%%%%%%%%%

\def\defbb#1{\expandafter\def\csname b#1\endcsname{\mathbb{#1}}}
\def\defcal#1{\expandafter\def\csname c#1\endcsname{\mathcal{#1}}}
\def\deffrak#1{\expandafter\def\csname frak#1\endcsname{\mathfrak{#1}}}
\def\defop#1{\expandafter\def\csname#1\endcsname{\operatorname{#1}}}
\def\defbf#1{\expandafter\def\csname b#1\endcsname{\mathbf{#1}}}

\makeatletter
\def\defcals#1{\@defcals#1\@nil}
\def\@defcals#1{\ifx#1\@nil\else\defcal{#1}\expandafter\@defcals\fi}
\def\deffraks#1{\@deffraks#1\@nil}
\def\@deffraks#1{\ifx#1\@nil\else\deffrak{#1}\expandafter\@deffraks\fi}
\def\defbbs#1{\@defbbs#1\@nil}
\def\@defbbs#1{\ifx#1\@nil\else\defbb{#1}\expandafter\@defbbs\fi}
\def\defbfs#1{\@defbfs#1\@nil}
\def\@defbfs#1{\ifx#1\@nil\else\defbf{#1}\expandafter\@defbfs\fi}
\def\defops#1{\@defops#1,\@nil}
\def\@defops#1,#2\@nil{\if\relax#1\relax\else\defop{#1}\fi\if\relax#2\relax\else\expandafter\@defops#2\@nil\fi}
\makeatother

%%%%%%%%%%%%%%%%%%%%%%%%%%%%
%%%% Benutzerteil: %%%%%%%%%%%%%%%%%
%%%%%%%%%%%%%%%%%%%%%%%%%%%%
\defbbs{ZHQCNPALRVW}
\defcals{ABOPQMNXYLTRAEHZKCFI}
\deffraks{abpijklgmnopqueRCB}
\defops{IVC, PGL,SL,mod,Spec,Re,Gal,Tr,End,GL,Hom,PSL,H,div,Aut,rk,Mod,R,T,Tr,Mat,Vol,MV,Res,Hur, vol,Z,diag,Hyp,hyp,hl,ord,Im,ev,U,dev,c,CH,fin,pr,Pic,lcm,ch,td,LG,id,Sym,Aut,Log,tw,irr,discrep,BN,NF,NC,age,hor,lev,ram,NH,av,app,mid,area,cyl,rank,thick,thin}
\defbfs{icuvzwp} % only use this for small letters, otherwise you go directly to hell!!!
%%%%%%%%%%%%%%%%%%%%%%%%%%%%
%%%%%%%%%%%%%%%%%%%%%%%%%%%%

\def\ep{\varepsilon}
\def\ve{\varepsilon}
\def\abs#1{\lvert#1\rvert}
\def\dd{\mathrm{d}}
\def\WP{\mathrm{WP}}
\def\inj{\hookrightarrow}
\def\eq{=}

\newcommand{\bfi}{{\bf i}}
\def\i{\mathrm{i}}
\def\e{\mathrm{e}}
\def\st{\mathrm{st}}
\def\ct{\mathrm{ct}}
\def\rel{\mathrm{rel}}
\def\odd{\mathrm{odd}}
\def\even{\mathrm{even}}
\def\icyl{i\mathrm{-cyl}}

\def\uC{\underline{\bC}}
\def\ol{\overline}
  
\def\Vrel{\bV^{\mathrm{rel}}}
\def\Wrel{\bW^{\mathrm{rel}}}
\def\twolev{\mathrm{LG_1(B)}}

\def\rk{\rho}% alias for the rank (m+1)/2

%%%%%%%%%%%%%% equations
\def\be{\begin{equation}}   \def\ee{\end{equation}}     \def\bes{\begin{equation*}}    \def\ees{\end{equation*}}
\def\ba{\be\begin{aligned}} \def\ea{\end{aligned}\ee}   \def\bas{\bes\begin{aligned}}  \def\eas{\end{aligned}\ees}
\def\={\;=\;}  \def\+{\,+\,} \def\m{\,-\,}

%%%%%%%%%%%%% moduli spaces
\newcommand*{\proj}{\mathbb{P}}
\newcommand{\IVCst}[1][\mu]{{\mathcal{IVC}}({#1})}
\newcommand{\barmoduli}[1][g]{{\overline{\mathcal M}}_{#1}}
\newcommand{\moduli}[1][g]{{\mathcal M}_{#1}}
\newcommand{\omoduli}[1][g]{{\Omega\mathcal M}_{#1}}
\newcommand{\oamoduli}[1][g]{{\Omega_1\mathcal M}_{#1}}
\newcommand{\modulin}[1][g,n]{{\mathcal M}_{#1}}
\newcommand{\omodulin}[1][g,n]{{\Omega\mathcal M}_{#1}}
\newcommand{\zomoduli}[1][]{{\mathcal H}_{#1}}
\newcommand{\barzomoduli}[1][]{{\overline{\mathcal H}_{#1}}}
\newcommand{\pomoduli}[1][g]{{\proj\Omega\mathcal M}_{#1}}
\newcommand{\pomodulin}[1][g,n]{{\proj\Omega\mathcal M}_{#1}}
\newcommand{\pobarmoduli}[1][g]{{\proj\Omega\overline{\mathcal M}}_{#1}}
\newcommand{\pobarmodulin}[1][g,n]{{\proj\Omega\overline{\mathcal M}}_{#1}}
\newcommand{\potmoduli}[1][g]{\proj\Omega\tilde{\mathcal{M}}_{#1}}
\newcommand{\obarmoduli}[1][g]{{\Omega\overline{\mathcal M}}_{#1}}
\newcommand{\obarmodulio}[1][g]{{\Omega\overline{\mathcal M}}_{#1}^{0}}
\newcommand{\otmoduli}[1][g]{\Omega\tilde{\mathcal{M}}_{#1}}
\newcommand{\pom}[1][g]{\proj\Omega{\mathcal M}_{#1}}
\newcommand{\pobarm}[1][g]{\proj\Omega\overline{\mathcal M}_{#1}}
\newcommand{\pobarmn}[1][g,n]{\proj\Omega\overline{\mathcal M}_{#1}}
\newcommand{\princbound}{\partial\mathcal{H}}
\newcommand{\omoduliinc}[2][g,n]{{\Omega\mathcal M}_{#1}^{{\rm inc}}(#2)}
\newcommand{\obarmoduliinc}[2][g,n]{{\Omega\overline{\mathcal M}}_{#1}^{{\rm inc}}(#2)}
\newcommand{\pobarmoduliinc}[2][g,n]{{\proj\Omega\overline{\mathcal M}}_{#1}^{{\rm inc}}(#2)}
\newcommand{\otildemoduliinc}[2][g,n]{{\Omega\widetilde{\mathcal M}}_{#1}^{{\rm inc}}(#2)}
\newcommand{\potildemoduliinc}[2][g,n]{{\proj\Omega\widetilde{\mathcal M}}_{#1}^{{\rm inc}}(#2)}
\newcommand{\omoduliincp}[2][g,\lbrace n \rbrace]{{\Omega\mathcal M}_{#1}^{{\rm inc}}(#2)}
\newcommand{\obarmoduliincp}[2][g,\lbrace n \rbrace]{{\Omega\overline{\mathcal M}}_{#1}^{{\rm inc}}(#2)}
\newcommand{\obarmodulin}[1][g,n]{{\Omega\overline{\mathcal M}}_{#1}}
\newcommand{\LTH}[1][g,n]{{K \overline{\mathcal M}}_{#1}}
\newcommand{\PLS}[1][g,n]{{\bP\Xi \mathcal M}_{#1}}

\DeclareDocumentCommand{\LMS}{ O{\mu} O{g,n} O{}}{\Xi\overline{\mathcal{M}}^{#3}_{#2}(#1)}
\DeclareDocumentCommand{\Romod}{ O{\mu} O{g,n} O{}}{\Omega\mathcal{M}^{#3}_{#2}(#1)}

\newcommand*{\Tw}[1][\Lambda]{\mathrm{Tw}_{#1}}  %twist group
\newcommand*{\sTw}[1][\Lambda]{\mathrm{Tw}_{#1}^s}  %simple twist group

\newcommand{\HH}{{\mathbb H}}
\newcommand{\MM}{{\mathbb M}}
\newcommand{\CC}{{\mathbb C}}
\newcommand{\RR}{{\mathbb R}}
\newcommand{\BM}{\overline{\mathcal M}}
\newcommand{\BcN}{\overline{\mathcal N}}
\newcommand{\calO}{\mathcal O}
\newcommand{\PP}{{\mathbb P}}
\newcommand{\calH}{{\mathcal H}}
\newcommand{\calM}{{\mathcal M}}
\newcommand{\calN}{{\mathcal N}}
\newcommand{\calC}{{\mathcal C}}

% Boldface letters
\newcommand{\bfa}{{\bf a}}
\newcommand{\bfb}{{\bf b}}
\newcommand{\bfc}{{\bf c}}
\newcommand{\bfd}{{\bf d}}
\newcommand{\bfe}{{\bf e}}
\newcommand{\bff}{{\bf f}}
\newcommand{\bfg}{{\bf g}}
\newcommand{\bfh}{{\bf h}}
\newcommand{\bfm}{{\bf m}}
\newcommand{\bfn}{{\bf n}}
\newcommand{\bfp}{{\bf p}}
\newcommand{\bfq}{{\bf q}}
\newcommand{\bft}{{\bf t}}
\newcommand{\bfP}{{\bf P}}
\newcommand{\bfR}{{\bf R}}
\newcommand{\bfU}{{\bf U}}
\newcommand{\bfu}{{\bf u}}
\newcommand{\bfx}{{\bf x}}
\newcommand{\bfz}{{\bf z}}

\newcommand{\bfl}{{\boldsymbol{\ell}}}
\newcommand{\bfmu}{{\boldsymbol{\mu}}}
\newcommand{\bfeta}{{\boldsymbol{\eta}}}
\newcommand{\bftau}{{\boldsymbol{\tau}}}
\newcommand{\bfomega}{{\boldsymbol{\omega}}}
\newcommand{\bfsigma}{{\boldsymbol{\sigma}}}

\newcommand{\wh}{\widehat}
\newcommand{\wt}{\widetilde}

\newcommand{\ps}{\mathrm{ps}}  
\newcommand{\CPT}{\mathrm{CPT}}  
\newcommand{\NCT}{\mathrm{NCT}}  
\newcommand{\RBD}{\mathrm{RBD}}
\newcommand{\ETD}{\mathrm{ETD}}
\newcommand{\OCT}{\mathrm{OCT}}
\newcommand{\SRT}{\mathrm{SRT}}
\newcommand{\RBT}{\mathrm{RBT}}
\newcommand{\Span}{\operatorname{Span}}
\newcommand{\vol}{\operatorname{vol}}
\newcommand{\emz}{\operatorname{emz}}

\newcommand{\tdpm}[1][{\Gamma}]{\mathfrak{W}_{\operatorname{pm}}(#1)}
\newcommand{\tdps}[1][{\Gamma}]{\mathfrak{W}_{\operatorname{ps}}(#1)}

\newlength{\halfbls}\setlength{\halfbls}{.8\baselineskip}

\newcommand*{\Teichmuller}{Teich\-m\"uller\xspace}

\DeclareDocumentCommand{\MSfun}{ O{\mu} }{\mathbf{MS}({#1})}
%% THIS ONE HAS NEITHER GAMMA NOR LAMBDA!!!
\DeclareDocumentCommand{\MSgrp}{ O{\mu} }{\mathcal{MS}({#1})}
\DeclareDocumentCommand{\MScoarse}{ O{\mu} }{\mathrm{MS}({#1})}
\DeclareDocumentCommand{\tMScoarse}{ O{\mu} }{\widetilde{\mathbb{P}\mathrm{MS}}({#1})}

\newcommand{\kmin}{\kappa_{(2g-2)}}
\newcommand{\ktop}{\kappa_{\mu_\Gamma^{\top}}}
\newcommand{\kbot}{\kappa_{\mu_\Gamma^{\bot}}}

\newcommand{\Gammap}{{\Gamma^+}}

\newcommand{\prodt}[1][j]{ t_{\lceil #1 \rceil}}
\newcommand{\prodtL}[1][\lceil L \rceil]{t_{#1}}

%%% shortform of matrices
\def\sm#1#2#3#4{\bigl(\smallmatrix#1&#2\\#3&#4\endsmallmatrix\bigr)} 
\def\mat#1#2#3#4{\begin{pmatrix}#1&#2\\#3&#4\\ \end{pmatrix}}   

\def\vn#1{\textcolor{black}{#1}}

%%%%%%%%%%%%%%%%%%%%%%%%%
%%%%%%%%%%%%%%%%%%%%%%%%%

\title[Area Siegel--Veech constants for affine invariant submanifolds]
      {Area Siegel--Veech constants for affine invariant submanifolds of REL zero}

\author{Dawei Chen}
\thanks{Research of D.C. is partially supported by National Science Foundation Grants DMS-2301030, DMS-2001040, Simons Travel Support for Mathematicians, and Simons Fellowship.}
\address{Department of Mathematics, Boston College, Chestnut Hill, MA 02467, USA}
\email{dawei.chen@bc.edu}

\author{Elise Goujard}
\address{
  Institut de Math{\'e}matiques de Bordeaux,
  Universit{\'e} de Bordeaux, 
351, cours de la Lib{\'e}ration,
F-33405 Talence \\
}
\email{elise.goujard@math.u-bordeaux.fr}

\author{Martin M\"oller}
\thanks{Research of M.M. is partially supported  
  by the DFG-project MO 1884/3-1 and the Collaborative Research Centre
TRR 326 ``Geometry and Arithmetic of Uniformized Structures.''}
\address{
Institut f\"ur Mathematik, Goethe--Universit\"at Frankfurt,
Robert-Mayer-Str. 6--8,
60325 Frankfurt am Main, Germany
}
\email{moeller@math.uni-frankfurt.de}

\begin{abstract}
We describe the principal boundary of an arbitrary affine invariant submanifold of REL zero in terms of level graphs of the multi-scale compactification of strata of Abelian differentials with prescribed orders of zeros. We show that the area Siegel--Veech constant of the affine invariant submanifold can be obtained by using volumes of the principal boundary strata. As an application, we prove the conjectural formula in \cite{CMSprincipal} that computes the area Siegel--Veech constant via intersection theory in the case of REL zero. In particular, the formula holds for strata of quadratic differentials with odd orders of zeros and for the gothic locus. We also explicitly describe the principal boundary components of the gothic locus and their individual contributions to the area Siegel--Veech constant.  
\end{abstract}

\maketitle
\tableofcontents

%%%%%%%%%%%%%%%%%%%%%%%%%%%%%%%
%%%%%%%%%%%%%%%%%%%%%%%%%%%%%%%%%%%%%%%%%%%%%%%%%%%%%%%%%%
\section{Introduction}
\label{sec:Intro}
%%%%%%%%%%%%%%%%%%%%%%%%%%%%%%%%%%%%%%%%%%%%%%%%%%%%%%%%%

For a partition $\mu = (m_1, \ldots, m_n)$ of $2g-2$, i.e., $\sum_{i=1}^n m_i  =  2g-2$, let $\omoduli[g,n](\mu)$ be the stratum of Abelian differentials $\omega$ on smooth and connected complex curves~$X$ of genus~$g$ that have labeled zeros whose orders are specified by $\mu$. Such differentials correspond to flat (or translation) surfaces with conical singularities (or saddle points) at the zeros, where the cone angle at each zero is $(m_i+1)2\pi$.  The study of differentials and flat surfaces is important for surface dynamics and moduli theory. In particular, there is a ${\rm GL}_2(\bR)$-action on $\omoduli[g,n](\mu)$ by affine transformations, which is called Teichm\"uller dynamics. The ${\rm GL}_2(\bR)$-orbit closures have a locally linear structure under the period coordinates of the ambient stratum (see \cite{esmi, esmimo}), and hence they are also called \emph{affine invariant submanifolds} (or {\em linear subvarieties} from a more algebraic viewpoint).    
\par
A number of intriguing dynamical invariants can be defined for the strata, and more generally, for affine invariant submanifolds. Among them, Siegel--Veech constants indicate the asymptotic growth rates of cylinders or saddle connections with bounded lengths for a generic flat surface parameterized in an affine invariant submanifold. Relatedly, Lyapunov exponents measure the separation rates of infinitesimally closed trajectories under the Teichm\"uller geometric flow.  Note that the sum of Lyapunov exponents and the {\em area} Siegel--Veech constant (i.e., counting cylinders weighted by their areas) determine each other (see \cite{ekz}).  
\par 
Our aim here is to compute the area Siegel--Veech constants through intersection theory on compactified affine invariant submanifolds. We set up some notation first.  
\par
Let $\cM$ be an affine invariant submanifold in a stratum of Abelian differentials $\omoduli[g,n](\mu)$. Suppose the tangent space of $\cM$ projects onto a subspace $A$ of absolute periods, with kernel $R$ of relative periods. Let $\dim_{\bC} A = a$ and $\dim_{\bC} R = r$, so that $\dim_{\bC} \cM = a + r =:m +1$, where $m$ is the projectivized dimension of $\cM$.   
\par 
Denote by $\bP\cM$ the projectivization of $\cM$ and by 
$\bP\ol{\cM}$ (the normalization of) the closure of $\bP\cM$ in the
multi-scale compactification of the projectivized stratum of Abelian differentials. Then $\dim_{\bC} \bP\ol{\cM} = m = (a-1) + r$.  Let $\xi$
be the first Chern class of the universal line bundle $\cO(1)$ on $\bP\ol{\cM}$, $\delta$
the boundary divisor class, and $\psi_i$ the cotangent line class associated to the $i$-th zero.  
 In \cite[Conjecture 4.3]{CMSprincipal}, the following conjectural formula was proposed for computing the area Siegel--Veech constant 
 $c_{\area}(\cM)$ via intersection theory: 
 \ba
 \label{eq:conjecture}
 c_{\area}(\cM) \stackrel{?} \= -\frac{1}{4\pi^2} \cdot \frac{ \int_{\bP\ol{\cM}}\xi^{a-2} \psi_{1} \cdots \psi_{r} \delta }{ \int_{\bP\ol{\cM}}\xi^{a-1} \psi_{1} \cdots \psi_{r}}\,. 
 \ea 
\par
 In this paper, we prove the conjectural formula in the case of REL zero, i.e., when $r = 0$.  
\par 
 \begin{theorem}
 \label{thm:main}
 Let $\cM$ be an affine invariant submanifold of REL zero and complex dimension $m+1$. Then the area Siegel--Veech constant of $\cM$ is equal to
 \ba
 \label{eq:c-rel0}
  c_{\area}(\cM) \= -\frac{1}{4\pi^2} \cdot \frac{ \int_{\bP\ol{\cM}}\xi^{m-1}\delta}{\int_{\bP\ol{\cM}}\xi^{m} }\,. 
  \ea
 \end{theorem}
 \par
 The sum of (nonnegative) Lyapunov exponents of $\cM$ is determined by $c_{\area}(\cM)$ via \cite[Formula (2.1)]{ekz}:
 $$ L(\cM) \= \frac{\pi^2}{3} c_{\area}(\cM)  + \frac{1}{12} \kappa_\mu $$
 where $\kappa_\mu = \sum_{i=1}^n m_i (m_i+2) / (m_i+1)$ for $\mu = (m_1, \ldots, m_n)$. The algebro-geometric counterpart of this formula is the relation of divisor classes $ \lambda = (\delta + \kappa)/12$
 on the moduli space of curves, where $\lambda$ is the first Chern class of the Hodge bundle and 
 $\kappa$ is the first Miller--Morita--Mumford class. Moreover, the relation of divisor classes 
 $ \kappa =  -\kappa_\mu \xi $ 
 holds on the multi-scale compactification of the strata of differentials modulo the boundary divisors $\delta_\Lambda$ with two-level graphs $\Lambda$ (see \cite[Proposition 6.3]{CCMkodaira}). Since 
 $\xi^{m-1} \delta_\Lambda = 0$ in the case of REL zero (see Lemma~\ref{le:non-horizontal}), combining these relations with Formula~\eqref{eq:c-rel0} thus leads to a companion formula for computing the sum of Lyapunov exponents of $\cM$ via intersection theory.   
 \par 
 \begin{cor} \label{intro:lyap}
  Let $\cM$ be an affine invariant submanifold of REL zero and complex dimension $m+1$. Then the sum of Lyapunov exponents of $\cM$ is equal to 
 \ba
 \label{eq:L-rel0}
  L(\cM) \=  - \frac{ \int_{\bP\ol{\cM}}\xi^{m-1}\lambda}{\int_{\bP\ol{\cM}}\xi^{m} }\,. 
  \ea
 \end{cor} 
 \par 
 The lowest dimensional affine invariant submanifolds are Teichm\"uller curves, which are of REL zero,  
 whose (unprojectivized) dimension is two, and whose corresponding ${\rm GL}_2(\bR)$-orbits are closed in $\omoduli[g,n](\mu)$. Prominent examples of higher-dimensional affine invariant submanifolds of REL zero include the minimal strata of Abelian differentials $\omoduli[g,1](2g-2)$, strata of quadratic differentials with odd orders of zeros only, and the gothic locus \cite{MMWgothic}. Formula~\eqref{eq:c-rel0} was previously known to hold for Teichm\"uller curves \cite{chenrigid, cmNV}, for the minimal strata $\omoduli[g,1](2g-2)$ \cite{SauvagetMinimal, CMSZ}, and for the principal strata of quadratic differentials with simple zeros \cite{CMSprincipal}. Concerning the gothic locus, the intersection numbers on the right-hand side of~\eqref{eq:c-rel0} were evaluated in \cite{chengothic} (without proving the correspondence to the area Siegel--Veech constant). Therefore, Theorem~\ref{thm:main} implies the following new advances in this direction. 
 \par
 \begin{cor} \label{cor:introgothic}
 Formula~\eqref{eq:c-rel0} holds for all strata of quadratic differentials with only odd orders of zeros and for the gothic locus. In particular, the area Siegel--Veech constant
  of the gothic locus is $49/(13\pi^2)$. Moreover, the gothic locus has four types of principal boundary graphs and their individual contributions to the area Siegel--Veech constant can be determined explicitly (see Proposition~\ref{prop:SVgothicCases}).  
 \end{cor}
 \par 
Now we explain the strategy of proving Theorem~\ref{thm:main}. First, we describe a set of level graphs (in the sense of \cite{LMS}) such that all non-zero contributions to Siegel--Veech constants are supported on a neighborhood of these graphs. We call these graphs \emph{principal graphs}, as they generalize the role played by the principal boundary in the sense of \cite{emz} for Abelian differentials (see also
\cite{GoujSV} for the case of quadratic differentials). Note that in our setting, principal graphs do not correspond to boundary divisors of the multi-scale compactification except for the case of Teichm\"uller curves, but to higher-codimensional boundary strata. The reason why they are nevertheless principal for volume computations is explained in Section~\ref{sec:codim}. Next, we derive an expression for $c_{\area}(\cM)$ in terms of volumes of the top level strata of principal graphs using flat surface geometry. Then we show that the right-hand side of Formula \eqref{eq:c-rel0} can also be reduced to the principal graphs via divisor relations on the multi-scale compactification of affine invariant submanifolds. Finally, we show that the resulting two expressions coincide with each other. 
\par   
The assumption of REL zero plays an essential role for defining a natural volume form on $\cM$, by using the area form of flat surfaces and the associated Hermitian metric on the $\calO(-1)$ bundle. If $\cM$ has nonzero REL, then the part of relative periods is not captured in the area form, and hence it is unclear how to canonically define a volume form on $\cM$. Nevertheless, once we can make sense of substituting relative periods by using $\psi$-classes to obtain the volume of $\cM$ via intersection theory, then we expect that the strategy in this paper can be adapted to prove \eqref{eq:conjecture} in full generality. On the other hand, as we have seen, in the case of REL zero the intersection-theoretic formula can already help us evaluate area Siegel--Veech constants, e.g., for the gothic locus. There are several other newly discovered affine invariant submanifolds of REL zero in \cite{EMMW}.  It would be interesting to analyze their principal graphs and then apply our result to evaluate their area Siegel--Veech constants.  We plan to treat these questions in future work. 
\par
This paper is organized as follows. In Section~\ref{sec:background}, we review basic properties of affine invariant submanifolds and their closures in the multi-scale compactification of strata of differentials. In Section~\ref{sec:PBlinMf}, we define and study the principal boundary components as well as their associated graphs for Siegel--Veech constants. In Section~\ref{sec:disintegration}, we reduce the computation for the area and cylinder Siegel--Veech constants of an affine invariant submanifold $\cM$ of REL zero to weighted sums of volumes of the principal boundary components mod out the volume of $\cM$. In Section~\ref{sec:intersection}, we show that the expression of the volume reduction obtained in Section~\ref{sec:disintegration} for the area Siegel--Veech constant coincides with that from the intersection calculation, thus proving Theorem~\ref{thm:main}. Finally, in Section~\ref{sec:example}, we provide a number of known examples of affine invariant submanifolds of REL zero. In particular, we describe the principal boundary components of the gothic locus and determine their respective contributions to the area Siegel--Veech constant. 
\par
\subsection*{Acknowledgements} The authors thank Vincent Delecroix, Samuel Grushevsky, Scott Mullane, Duc-Manh Nguyen, Adrien Sauvaget, Johannes Schwab, Guillaume Tahar, Fei Yu, and Anton Zorich for relevant discussions. This work was initiated at the CIRM conference ``Combinatorics, Dynamics and Geometry on Moduli Spaces'' in 2022 and completed at the MATRIX program ``Teichm\"uller Theory and Flat Structures'' in 2025. The authors are grateful to the organizers of these events for their invitations and hospitality.  
 
%%%%%%%%%%%%%%%%%%%%%%%%%%%%%%%

%%%%%%%%%%%%%%%%%%%%%%%%%%%%%%%

%%%%%%%%%%%%%%%%%%%%%%%%%%%%%%%%%%%%%%%%%%%%%%%%%%%%%%%%%%
\section{Affine invariant submanifolds and their closures}
\label{sec:background}
%%%%%%%%%%%%%%%%%%%%%%%%%%%%%%%%%%%%%%%%%%%%%%%%%%%%%%%%%

%%%%%%%%%%%%%%%%%%%%%%%%%%%%%%%%%%%%%%%%%%%%%%%%%%%%%%%%%%
\subsection{Affine invariant submanifolds}
\label{sec:linMf}
%%%%%%%%%%%%%%%%%%%%%%%%%%%%%%%%%%%%%%%%%%%%%%%%%%%%%%%%%

Let $\mu = (m_1,\ldots,m_n)$ be a tuple of nonnegative integers with $\sum_{i=1}^n m_i = 2g-2$. Denote by $\omoduli[g,n](\mu)$ the moduli space of Abelian differentials $(X, \omega)$ of type~$\mu$ with the~$n$ zeros $Z(\omega)$ of the differential $\omega$ labeled as marked points in the underlying surface $X$. By the fundamental results of \cite{esmi, esmimo}, the closure of any $\GL_2(\bR)$-orbit in $\omoduli[g,n](\mu)$ locally consists of  a finite union of linear subspaces in period coordinate charts. For this reason, we define an {\em affine invariant submanifold} $\cM$ of $\omoduli[g,n](\mu)$ as an algebraic stack with a map $\cM \to \omoduli[g,n](\mu)$ which is the {\em normalization} of a $\GL_2(\bR)$-orbit closure as its image. We write $\bP \cM$ for the projectivization of~$\cM$ modulo scaling of $\bC^{*}$.
\par
We recall a few fundamental properties of affine invariant submanifolds
(see \cite{WrightNumFd}). The projection 
\bes
p\colon H_1(X,Z(\omega),\bR) \to H_1(X,\bR)
\ees
from relative homology to absolute homology induces a short exact sequence 
\bes
0 \to \ker(p) \cap T\cM \to T\cM \to p(T\cM) \to 0
\ees
where $T\cM$ is the tangent bundle to $\cM$. The image $p(T\cM)$ is of {\em even} 
dimension~
$$a \= 2\rk\,.$$ 
We refer to $\rk$ as the \emph{rank} of~$\cM$
and denote by $r = \dim(\ker(p) \cap T\cM)$ the \emph{rank of relative periods}
of~$\cM$. If $r=0$, we say that $\cM$ has \emph{REL zero}. The tangent bundle~$T\cM$
is defined over a number field. There is a minimal such number field, called 
the \emph{affine field of definition $k(\cM)$}.
\par
The differential $\omega$ induces a flat metric on $X$ with conical singularities at the zeros of $\omega$. For this reason,  $(X, \omega)$ is also called a \emph{flat surface} (or a {\em translation surface}).   Following \cite{WrightCyl} we say that a collection~$C_1,\ldots, C_t$ of cylinders on a flat surface $(X,\omega)$ in~$\cM$ is \emph{$\cM$-parallel}, if the core curves of the cylinders are collinear and stay collinear
in a neighborhood of the point parameterizing $(X,\omega)$ in~$\cM$.
\par

%%%%%%%%%%%%%%%%%%%%%%%%%%%%%%%%%%%%%%%%%%%%%%%%%%%%%%%%%%
\subsection{The multi-scale compactification}
\label{sec:MSD}
%%%%%%%%%%%%%%%%%%%%%%%%%%%%%%%%%%%%%%%%%%%%%%%%%%%%%%%%%

Let $\bP\LMS$ be the moduli space of projectivized multi-scale differentials constructed in \cite{LMS}, which compactifies the moduli space $\bP\omoduli[g,n](\mu)$.  In what follows we recall some properties about the boundary of $\bP\LMS$. 
\par
The boundary strata of $\bP\LMS$ are indexed by level graphs. A \emph{level
graph}~$\Gamma$ consists of a stable graph (of the underlying pointed stable curve) together with a weak total order on the vertices, usually given by a
level function normalized to have top level zero, and an enhancement~$\kappa_e
\geq 0$ associated to each edge $e$. An edge is called \emph{horizontal} if
it starts and ends at the same level, and called \emph{vertical} if it joins two vertices at different levels. Moreover, $\kappa_e = 0$ if and only if the edge $e$ is horizontal. We let
\bes
\ell_i \= \lcm(\kappa_e : \text{$e$ crosses the passage above level~$i$})\,.
\ees  
\par
The union of vertices at any given level of a level graph defines
the type of a \emph{generalized stratum}, that is a stratum of possibly 
meromorphic differentials on a possibly disconnected complex curve, which is moreover
constrained by residue conditions. Here the components of the curve
are given by the vertices of~$\Gamma$ at the given level, and the orders
of zeros and poles of the differential are determined by the half-edge labels
and the enhancements. We refer to \cite[Section~4]{cmz} for more 
details and for a way to encode the residue conditions. A collection
of differentials in the generalized strata for each level gives a
\emph{twisted differential} compatible with~$\Gamma$. A \emph{multi-scale
differential} (compatible with~$\Gamma$) is
a twisted differential together with another discrete datum, a collection
of {\em prong-matchings}, and with the differentials on
lower level defined up to projectivization. Given a level graph~$\Gamma$, we
denote by $D_\Gamma^\circ$ the boundary stratum of all multi-scale
differentials compatible with~$\Gamma$ and we denote by~$D_\Gamma$ its closure.
\par
Fix a level graph~$\Gamma$ with $L$ levels below zero. Then the boundary
divisor decomposes up to commensurability into products of generalized strata
of differentials. Here the generalized strata~$B_\Gamma^{[i]}$ allow for
disconnected curves and the differentials may be constrained by residue
conditions. Commensurability means the existence of the following diagram
(see \cite[Diagram~(5)]{chernlinear})
\be \label{dia:covboundary} \begin{tikzcd}
&D_{\Gamma}^{s}\arrow{dl}[swap]{p_{\Gamma}} \arrow{dr}{c_{\Gamma}} && \\
\prod_{i=-L}^{0} B_\Gamma^{[i]} =: {B}_{\Gamma} &  & D_\Gamma
\end{tikzcd}
\ee
where~$p_\Gamma$ and~$c_\Gamma$ are finite maps when restricted to the
preimage of~$D_\Gamma^\circ$, starting at some space $D_{\Gamma}^{s}$ (referring to a ``simple'' twist group) whose precise definition is not needed here.
\par
The generalized strata~$B_\Gamma^{[i]}$ at each level are smooth stacks of some dimension $N_0$ resp.\ of dimension $N_i-1$ for $i=-1,\ldots,-L$ (duo to level-wise projectivization below the top level). By construction we have the relation
\be
\sum_{i=0}^{-L} N_i \= N \coloneqq\dim \LMS
\ee
among the (unprojectivized) dimensions of the level strata and the ambient
moduli space. For graphs with only two levels we often write
$N^\top \coloneqq N_0$ and $N^\bot \coloneqq N_{-1}$.
\par

%%%%%%%%%%%%%%%%%%%%%%%%%%%%%%%%%%%%%%%%%%%%%%%%%%%%%%%%%%
\subsection{The closures of affine invariant submanifolds}
\label{sec:ClosureLinMfd}
%%%%%%%%%%%%%%%%%%%%%%%%%%%%%%%%%%%%%%%%%%%%%%%%%%%%%%%%%

Given an affine invariant submanifold $\cM$, we denote by $\ol{\cM}$ the normalization of the closure of the image of $\cM$ as a substack of $\LMS$. The shape of local equations cutting out~$\ol{\cM}$ in $\LMS$ has been determined in \cite{BDG}. We summarize their results and adapt  to our case of affine invariant submanifolds of REL zero.
\par
Local defining equations of~$\cM$ can be interpreted as homology classes. We say that a horizontal node is \emph{crossed} by an equation, if the corresponding
vanishing cycle has non-trivial intersection with the homology class of the equation.
The horizontal nodes are partitioned into \emph{$\cM$-cross-equivalence
classes} by simultaneous appearance in equations of~$\cM$, i.e.,\ a collection
of nodes belongs to the same $\cM$-cross-equivalence class, defined as follows. If there is an equation crossing a set of nodes, but no equation crosses a proper subset of them, then the nodes are called \emph{$\cM$-cross-related}. The equivalence relation generated by this relation is $\cM$-cross-equivalence.
A starting observation \cite[Theorem~1.1]{BDG} is that $\omega$-periods of
the vanishing cycles in an $\cM$-cross-equivalence class are proportional.
In particular, horizontal nodes in the same $\cM$-cross-equivalence class
are on the same level.     
\par
Given a level graph $\Gamma$ with $L$ levels below zero, we briefly recall the coordinates near the corresponding boundary stratum~$D_\Gamma$. They consist first of \emph{level coordinates} $t_i$ for $i=-1,\ldots,-L$ that record the scaling of the $i$-th level. Second, for each horizontal edge at level~$i$, there is a coordinate $q^{[i]}_j$ that records the opening up of the corresponding node. Third, the remaining coordinates are period coordinates of differentials at the various levels. (At this stage we are glossing over the details where the latter are not always well-defined. To remedy this, one uses perturbed period coordinates \cite{LMS, CMZEuler}, or log period coordinates \cite{BDG}. However we will not need those details in this paper.)
\par
In \cite{BDG}, it shows that the equations cutting out the image of $\ol{\cM}$
in~$\LMS$ can be generated by the following two types of equations.
First, there are \emph{level-wise equations} among the period coordinates.
They are linear and involve only coordinates at one level. Second, there
are \emph{horizontal crossing equations}. These are multiplicative
equations among the $q^{[i]}_j$ for a given level~$i$. In particular, there
are no relations among the $t_i$. For the case of REL zero, we can make the
equations involving the~$q^{[i]}_j$ and the vanishing cycles even more
precise as follows. 
\par
\begin{theorem}[{\cite[Theorem 1.10]{BDG}}] \label{thm:BDG1_10}
Let $\cM$ be an affine invariant submanifold of REL zero. Then the space of horizontal crossing equations of the image of $\ol{\cM}$ is generated by equations
of the form $(q^{[i]}_{j_1})^a = (q^{[i]}_{j_2})^b$. Moreover, the periods of the vanishing cycles of the corresponding nodes are proportional. In particular, the normalization $\ol{\cM}$ is smooth.
\end{theorem} 
\par
In this theorem the exponents $a$ and $b$ are a priori real numbers just quoting
\cite{BDG}. We will show in Lemma~\ref{le:commensurable} that they are in
fact integers.
\par
\begin{proof} 
The corresponding result for $\cM$ in the minimal stratum $\omoduli[g,1](2g-2)$ (as a standard example of REL zero) was stated in \cite[Theorem 1.10]{BDG}. Indeed the proof only relied on the assumption of REL zero (i.e., the maps $\iota$ and $u$ are isomorphisms in \cite[Section 6.1]{BDG}). Therefore, the same proof works for all affine invariant submanifolds of REL zero. 
\end{proof}
\par
As a consequence of the structure of equations and the commensurability
diagram~\eqref{dia:covboundary}, we obtain a similar commensurability
decomposition of the boundary stratum $\calM_\Gamma$, defined as the
intersection of $\ol{\calM}$ with the preimage of~$D_\Gamma$, into generalized
linear submanifolds $\calM_{\Gamma}^{[i]}$. Restricted to the open
set~$\calM_\Gamma^\circ$ over~$D_\Gamma^\circ$, we have the following diagram
\be \label{dia:calMboundary} \begin{tikzcd}
&\calM_{\Gamma}^{s}\arrow{dl}[swap]{p_{\Gamma,\calM}} \arrow{dr}{c_{\Gamma,\calM}} && \\
\prod_{i=-L}^{0} \calM_\Gamma^{[i]}  &  & \calM_\Gamma\,
\end{tikzcd}
\ee
for some space $\calM_{\Gamma}^{s}$ (see \cite[Diagram~(13)]{chernlinear}). In the case of two levels, i.e., $L=1$, we write $\calM_\Gamma^\top := \calM_\Gamma^{[0]}$ and $\calM_\Gamma^\bot := \calM_\Gamma^{[-1]}$ for the two levels of the decomposition up to commensurability. In this case, $\ell = \ell_1 = \lcm(\kappa_e)$, where~$e \in E(\Gamma)$ runs over all (vertical) edges. We let $N_\calM^\top = \dim \calM_\Gamma^\top$ and $N_\calM^\bot = \dim
\calM_\Gamma^\bot + 1$ be the dimensions of the unprojectivized level strata.
These dimensions are related by
\be
N_\calM^\top + N_\calM^\bot \= m+1\,
\ee
where $m$ is the projectivized dimension of  $\calM$. 
\par
If we moreover restrict to a component of~$\calM_\Gamma^\circ$ (we do this tacitly here and will introduce notation for them along with additional data in Section~\ref{sec:coordprinc}), we find specifically (see \cite[Lemma~3.6]{chernlinear}) that 
\be \label{eq:ratio_degrees}
\deg(c_{\Gamma}) \= \Aut_{\calM}(\Gamma)\,, \qquad 
\deg(p_{\Gamma}) \= \frac{K^\calM_\Gamma }{\ell_\Gamma}\,.
\ee
Here $\Aut_{\calM}(\Gamma)$ is the subgroup of~$\Aut(\Gamma)$ whose induced
action on a neighborhood of~$D_\Gamma$ preserves the image of $\calM$ in the
ambient stratum of differentials.  Moreover, $K_\Gamma^{\cM}$ is the number of reachable
prong-matchings with respect to~$\cM$,  which is described in
\cite[Theorem 1.3]{chernlinear}. 
\par

%%%%%%%%%%%%%%%%%%%%%%%%%%%%%%%

%%%%%%%%%%%%%%%%%%%%%%%%%%%%%%%%%%%%%%%%%%%%%%%%%%%%%%%%%%%%

%%%%%%%%%%%%%%%%%%%%%%%%%%%%%%%%%%%%%%%%%%%%%%%%%%%%%%%%%%
\section{Siegel--Veech constants and principal graphs}
\label{sec:PBlinMf}
%%%%%%%%%%%%%%%%%%%%%%%%%%%%%%%%%%%%%%%%%%%%%%%%%%%%%%%%%

%%%%%%%%%%%%%%%%%%%%%%%%%%%%%%%%%%%%%%%%%%%%%%%%%%%%%%%%%%
\subsection{Siegel--Veech constants}
\label{sec:SV}
%%%%%%%%%%%%%%%%%%%%%%%%%%%%%%%%%%%%%%%%%%%%%%%%%%%%%%%%%

Siegel--Veech constants measure the growth rates of the numbers of saddle connections or cylinders (with bounded lengths or widths) on a flat surface, possibly with additional constraints on the position of zeros relative to the saddle connections or cylinders. More precisely, these quantities all exhibit a quadratic growth rate and the Siegel--Veech constant is the leading term of large length asymptotics. Siegel--Veech constants can alternatively be defined as the comparison coefficient between two integrals. We will take this as our definition (see~\eqref{eq:SVbasic} below and refer to \cite{eskinmasur,emz,cmz} for the conversion of this coefficient into growth rate asymptotics).
\par
We start by recalling the notion of Siegel--Veech transform.
Let $V = V(X,\omega) \subset \bR^2$ be a function that associates with a flat
surface a subset in $\bR^2$ with (real) multiplicities, satisfying the 
Siegel--Veech axioms (see Section~2 in \cite{eskinmasur}). These axioms 
are roughly the $\SL_2(\bR)$-equivariance, the quadratic growth rate of $V$, 
and an integrability condition (see the examples below). For any function
$\chi\colon \bR^2 \to \bR$ we denote by $\widehat{\chi}$ the Siegel--Veech transform
with respect to $V$ as follows: 
\be \label{eq:SVtransform}
\widehat{\chi}(X,\omega) \= \sum_{v \in V(X,\omega)} \chi(v)\,.
\ee 
Let $\nu$ be a finite $\SL_2(\bR)$-invariant measure on a subset of the
hypersurface $\oamoduli(\mu)$ of area-one flat surfaces, 
whose support we denote by~$\cM_1$. By the fundamental results of
\cite{esmi} and \cite{esmimo},  this support is actually the area-one
hypersurface in an affine invariant submanifold. The results of Veech and  Eskin--Masur
(\cite{veech98}, \cite{eskinmasur}) jointly imply that for appropriate
$V$ and $\nu$ there is a \emph{Siegel--Veech constant} $c(\nu,V)$, such that
for all functions $\chi$ we have
\be \label{eq:SVbasic}
\frac{1}{\nu(H)} \int_{\cM_1} \widehat{\chi} d\nu \= c(\nu,V)  \int_{\bR^2} 
\chi dxdy\,. 
\ee
Note that rescaling~$\nu$ by a scalar factor leaves the Siegel--Veech constant $c(\nu,V)$
unchanged.
\par
In this paper we consider the following types of Siegel--Veech constants. If~$V$ is the vector function of all core curves of cylinders with weight one, we obtain the \emph{cylinder Siegel--Veech constant} $c_{\cyl}$. Taking again the core curves of cylinders, but with weight given by the area of the cylinder, we obtain the \emph{area Siegel--Veech constant} $c_{\area}$.

%%%%%%%%%%%%%%%%%%%%%%%%%%%%%%%%%%%%%%%%%%%%%%%%%%%%%%%%%%
\subsection{Volume normalization}
\label{sec:VolNorm}
%%%%%%%%%%%%%%%%%%%%%%%%%%%%%%%%%%%%%%%%%%%%%%%%%%%%%%%%%

By \cite{esmi} and \cite{esmimo}, every $\bR$-linear submanifold~$\cM$ comes
with a unique (up to scale) $\GL_2(\bR)$-invariant ergodic volume form, whose
restriction to the (real) hypersurface $\cM_1$ of area-one flat surfaces 
 is the $\nu$ introduced above. Our
goal here is to fix a scale~$\nu_\cM$ for the measure of all these
linear submanifolds, derived from the mere existence of the area form
on the flat surface. The following is a summary of \cite[Section~6.1]{NguVol} for this purpose, for which we also assume \emph{REL zero in this subsection}.
\par
Recall that $\cM$ has complex dimension~$m+1$. There are two ways to define $\nu_\cM$. By \cite{AEM} the restriction of the area form~$h$ induced by the flat metric to the tangent space $T\cM$ is a symplectic form. Consequently,
taking the imaginary part~$\vartheta$ of this symplectic form, then $\vn{(\bfi \vartheta)}^{m+1}/(m+1)!$ is a volume form. In a local chart where~$\cM$ is modeled on the vector space~$V$, the image under period coordinates is actually in the orthant~$V^+$ where~$h$ is positive. There we can write 
\be \label{eq:integrationcoordinates}
\vartheta \= \frac{{\bfi}}{2} \Big( \sum_{i=1}^{(m+1)/2} d{A_i} \wedge d\bar{A_i}
-  d{B_i} \wedge d\bar{B_i} \Big)\,, 
\ee
where $a_i = \int_{\alpha_i} \omega$ and $b_i = \int_{\beta_i} \omega$ are 
local coordinates obtained as periods of a \emph{symplectic}
basis $(\alpha_i, \beta_i)$ of the (real) tangent space~$T_\cM$, and 
\be \label{eq:symptoimag}
A_i \= \frac12(a_i - \bfi b_i)\,, \quad B_i \= \frac12(a_i + \bfi b_i)\,.
\ee
This defines the auxiliary \emph{intersection volume form}
\be \label{eq:defnuint}
\nu_{\rm int} = \frac{\vn{(\bfi \vartheta)}^{m+1}}{(m+1)!}
\ee
on $V$. For a measurable subset $B \subset \cM_1$
that fits in one period coordinate chart, we denote by $C(B)$ the cone in~$V^+$ bounded
by~$B$. We define the \emph{natural normalization of the volume form on~$\cM$}
to be 
\be \label{eq:defmeasure}
\nu_\cM(B) = \nu_{\rm int}(C(B))
\ee
and extend to general~$B$ by additivity. Similarly we define a measure $\ol{\nu}$ on~$\bP\cM$ by
\be \label{eq:defprojmeasure}
\ol{\nu}_\cM(B) = \nu_{\rm int}(\wh{C}(B))
\ee
where the cone $\wh{C}(B)$ is
the intersection of the preimage of~$B$ in~$\cM$ with the locus $\cM_{\leq 1}$ of flat surfaces of area $\leq 1$.
\par
For the second definition, consider the $(1,1)$-form 
$$\varpi = \frac{-1}{2 \pi \bfi}
\partial \bar{\partial} \log(h)\,.$$ 
Its top wedge power defines another
volume form. A local computation (\cite[Lemma~6.1]{NguVol}) shows that
\bes
\ol{\nu}_\cM \=  (-1)^m \cdot \frac{2\pi}{2^{m+1}}\cdot \frac{\vn{\bfi^{m+1}}}{(m+1)!} \cdot (-2\pi \varpi)^m  \=
\frac{\vn{(\bfi\pi)}^{m+1}}{(m+1)!} \cdot \varpi^m\,. 
\ees
On the other hand, the cohomology class of $\varpi$ is precisely $c_1(\cO(1))
= \xi$ (see \cite{SauvagetMinimal,goodmetric}). We thus obtain that
\be \label{eq:vol-int}
\Vol(\cM) \,:=\, (2m+2)
\ol{\nu}_\cM(\bP\cM) \= (2m+2) \nu_\cM(\cM_1) \= \frac{2}{m!}\cdot (\vn{\bfi}\pi)^{m+1}\cdot \int_{\bP \ol{\cM}} \xi^m\,.
\ee
\par 
Finally, it is sometimes convenient to define a volume element 
$\nu_{\rm dis}$ on the hypersurface $\cM_1$ of area-one flat surfaces by disintegration, i.e., by
the requirement 
\be \label{eq:defnudis}
d\nu_{\rm int} = r^{2m+1} dr d\nu_{\rm dis}\,
\ee
where $\dim_\RR \calM = 2m+2$. These volume forms are related by
\be \label{eq:volrelation}
\nu_{\rm dis}(U) \= \dim_\RR \calM \cdot \nu_{\calM}(U)\,, \qquad
\text{hence} \qquad  \Vol(\cM) \= \nu_{\rm dis}(\calM_1)\,
\ee
(see \cite[Section~2.2]{emz}).
\par
We remark that when $\cM$ is defined over $\bQ$, e.g., for strata of differentials and the gothic locus, one can use alternative volume normalizations by 
choosing a unit cube under period coordinates (see \cite{emz}, \cite{GoujSV} and \cite{Torres-Teigell}). Nevertheless, the volume normalization by using the area form $\varpi$ in the above seems more intrinsic, as it can also work when the period coordinates do not span a lattice. 

%%%%%%%%%%%%%%%%%%%%%%%%%%%%%%%%%%%%%%%%%%%%%%%%%%%%%%%%%%
\subsection{Principal graphs}
\label{sec:principalgraphs}
%%%%%%%%%%%%%%%%%%%%%%%%%%%%%%%%%%%%%%%%%%%%%%%%%%%%%%%%%

We will show that boundary strata with the following type of level graphs $\Gamma$ support all the contributions to cylinder and area Siegel--Veech constants. We remark that the  notion below does not only depend on the graph, but also on the top and bottom strata $\cM^\top$ and $\cM^\bot$ that potentially appear in the boundary intersection of~$\cM$ with~$D_\Gamma$. The main result of this subsection holds for general REL.
\par
\begin{definition} \label{def:principal}
{\rm (1)} Suppose the affine invariant submanifold~$\cM$ has REL zero and rank one, i.e., it is
a Teichm\"uller curve.  In this case a level graph~$\Gamma$ is called a~\emph{principal
graph (with respect to $\cM$}), if it has one level and its horizontal edges
form one $\cM$-parallelity class and one $\cM$-cross-equivalence class.
\par 
{\rm (2)} Suppose the affine invariant submanifold~$\cM$ has rank at least two.  A level graph~$\Gamma$ is called a~\emph{principal graph (with respect to~$\cM$)}, if
\begin{itemize}
\item[i)] the level graph $\Gamma$ has exactly two levels,
\item[ii)] the top level of~$\Gamma$  has no horizontal edges, 
\item[iii)] the unprojectivized bottom level dimension $N_{\Gamma}^\bot = 1$, and
\item[iv)] there is exactly one $\cM$-parallelity class of horizontal edges on
  the lower level.
\end{itemize}
If~$\cM$ has REL zero, then we moreover require that
\begin{itemize}
\item[(R)] the horizontal edges form one~$\cM$-cross-equivalence class.
\end{itemize}
\end{definition}
\par
We remark that if~$\cM$ has nonzero REL, then condition~(R) has to be modified, and we do not address the details here since our main application in this paper is for the case of REL zero.  
\par
Additionally, we do not claim that~$D_\Gamma$ with any such~$\Gamma$ has nonempty intersection with  $\bP \ol{\cM}$. If this is not the case, the contribution of~$D_\Gamma$ to the
concerned Siegel--Veech constants will be zero, and conversely, we show that the
contribution is non-zero otherwise. Additionally, in the case of rank one, i.e., for Teichm\"uller curves,  there do not exist boundary points with two-level principal graphs for the dimension reason.
\par
In the case of REL zero, the next lemma ensures that the boundary defined by principal graphs stays within the scope of affine invariant submanifolds for which we can give a volume normalization.
\par
\begin{lemma} \label{le:rankdrop}
Suppose the affine invariant submanifold~$\cM$ has REL zero and $\rank(\cM)$ is at least two. If~$\Gamma$ is a graph satisfying i)--iv) in Definition~\ref{def:principal}, then the top level $\cM_\Gamma^\top$ also has REL zero, and $\rank(\cM_\Gamma^\top) = \rank(\cM)-1$. In particular, $\Gamma$ satisfies~(R) and hence it is a principal graph in the case of REL zero. 
\end{lemma}
\par
\begin{proof} If $\rank(\cM_\Gamma^\top) > \rank(\cM)-1$, then $\dim \cM_\Gamma^\top > 2\rank(\cM)- 2$. Since the period around the node
for one of the horizontal edges on lower level as well as the level
coordinate~$t_{-1}$  contribute additionally to the dimension of~$\cM$, we
obtain $\dim \cM > 2\rank(\cM)$, which leads to a contradiction.
\par
On the other hand, if $\rank(\cM_\Gamma^\top) < \rank(\cM) - 1$, then there are
at least two symplectic pairs in the tangent space to~$\cM$ that do not degenerate
to absolute periods on top level. Moreover, relative periods on top level only
contribute to a subspace on which the intersection form is trivial. Consequently,
there is a rank-two subspace of $T_\cM$ that degenerates to lower level on
which the intersection form is trivial. A principal graph has
$N_{\Gamma}^\bot = 1$, which is necessarily generated by a curve around the node
for one of the horizontal edges, and the periods created by opening the
horizontal nodes intersect this generator non-trivially. Therefore, there
can be at most one symplectic pair that degenerates to the lower level, which again leads to a contradiction.
\par 
Finally, since $\dim \cM_\Gamma^\top = 2 \rank(\cM_\Gamma^\top)$, it follows that $\cM_\Gamma^\top$ has REL zero. If there are more than one $\cM$-cross-equivalence classes on lower level, then the codimension of $\cM_\Gamma^{\top}$ would be bigger than two and hence has rank less than $\rank (\cM) - 1$, leading to a contradiction. Therefore, $\Gamma$ satisfies~(R) and it is a principal graph.  
\end{proof}
\par
We now prove a preliminary version of our goal.
\par
\begin{prop} \label{prop:locSV}
There exist \emph{local Siegel--Veech constants} $c_\star(\calM,\Gamma) \geq 0$
such that
\be \label{eq:Gammasum}
c_{\star}(\cM)\= \sum_{\Gamma~\rm{principal}} c_\star(\calM,\Gamma)
\ee
for $\star \in\{\area, \cyl\}$. \par
Suppose moreover that $\calM$ has REL zero. Then $c_\star(\calM,\Gamma) > 0$ if and only if $D^\circ_\Gamma \cap \bP \ol{\cM} \neq \emptyset$.
\end{prop}
\par
\begin{proof} Let $\cM_1^\ep$ be the locus in $\cM$ parameterizing flat surfaces $S$ 
of area~one with a cylinder of perimeter~$\leq \ep$.  As an immediate consequence
of the Siegel--Veech formula (valid for any affine invariant
submanifold), we obtain as in \cite[Proposition~3.3]{emz} that
\be \label{eq:sv-limit}
c_{\star}(\cM)\= \lim_{\ep \to 0}\frac{1}{\pi\varepsilon^2}\cdot \frac{\Vol_{\star}(\cM_1^\ep)}{\Vol(\cM)}\,
\ee
where $\Vol_{\star}(\cM_1^\ep)$ is defined as the weighted volume
\[\Vol_{\star}(\cM_1^\ep)= \dim_{\bR}\cM \cdot \int_{\cM_1^\ep} W_{\star}(S)
d\nu_\cM(S)\]
with the weights $W_\star$ depending only on the family of $\cM$-parallel cylinders
in~$S$. Explicit formulas for the weights are given in~\eqref{eq:weights}
below.
\par
Generalizing the strategy in \cite{emz}, we define $\cM_1^{\ep, \thick}$ as the
subset of $\cM_1^\ep$ consisting of surfaces with the following properties: 
\begin{itemize}
\item There is exactly one family $\cC = \{C_j\}_{j=1}^q$ of $\cM$-parallel
cylinders having at least one cylinder of perimeter~$\leq \ep$. We order the
cylinders increasingly by the length of their waist curves $w(C_i)$ and set
$W = w(C_q)/w(C_1)$.
\item All saddle connections at the boundary of any of the $C_j$ are 
$\cM$-parallel.
\item The length of any saddle connection not at the boundary of any of the~$C_j$ is larger than $3W\ep$.
\end{itemize}
We let $\cM_1^{\ep, \thin}$ be the complement of the thick part in $\cM_1^\ep$.
We claim that \cite[Theorem~1.3]{D} implies that
\[\Vol_{\ast}(\cM_1^\ep) \= \Vol_{\ast}(\cM_1^{\ep, \thick})+o(\ep^2)\,.\]
In fact the weight function is bounded by the number of cylinders, which is
bounded by $3g-3+n$ and Dozier has shown that the volume of the subset
parametrizing surfaces with two non-$\cM$-parallel saddle connections of lengths
$\ep_1$ and~$\ep_2$ is $O(\ep_1^2\ep_2^2)$.
\par
Assume from now on that $\mathrm{rank}(\cM)>1$. We now associate with every surface~$(X,\omega)$ in $\cM_1^{\ep, \thick}$ a principal graph~$\Gamma$ as follows. Consider the tubular neighborhood~$T \subset X$ consisting of all points that have distance~$\leq \ep$ to a point in~$\cC$. The rank
hypothesis and the last condition of the thick part together ensure that
$X \setminus T$ is non-empty. The vertices on top level of~$\Gamma$ correspond to the components of $X \setminus T$, and the vertices on lower level correspond to the components of~$T$ {with the core curves of all the cylinders in~$\calC$ removed}. The vertical edges of~$\Gamma$ correspond to the boundary components of the closure of~$X \setminus T$, decorated with $\kappa_e \geq 0$ where $2\pi \kappa_e$ is the total curvature of the boundary component for $e$. The legs of~$\Gamma$ are given by the zeros of~$\omega$ with obvious adjacency to the vertices. Finally, the horizontal edges of $\Gamma$ correspond to pinching the core curves of the cylinders in~$\calC$.
\par
To show that~$\Gamma$ is indeed a principal graph, we remark that all saddle
connections in~$T$ are $\cM$-parallel to the core curve of some~$C_i$ by the
first condition in the definition of~$\cM_1^{\ep, \thick}$ and thus $N_{\cM_\Gamma}^\bot = 1$.
The other conditions are obvious by construction, except for the additional 
condition (R) in the case of REL zero. If there are more than one $\calM$-cross-equivalence
classes on lower level, then the codimension of $\overline{\calM} \cap D_\Gamma$ is bigger than two. Given that the horizontal edges are on lower level, this contradicts Lemma~\ref{le:rankdrop}.
\par
We call the level graph~$\Gamma$ the \emph{configuration} associated with
$(X,\omega) \in \cM_1^{\ep, \thick}$ and write $\cM_{1,\Gamma}^{\ep}$ for the subset
of surfaces with configuration~$\Gamma$. With the definition
\be\label{eq:cGamma}
c_{\star}(\cM,\Gamma)\= \lim_{\ep \to 0} \frac{1}{\pi\varepsilon^2} \cdot \frac{\Vol_{\star}(\cM_{1,\Gamma}^\ep)}
{\Vol(\cM)}\,,
\ee
the non-negativity claim is obvious and the claimed identity~\eqref{eq:Gammasum}
follows since we only neglected terms of higher order in~$\ep$. The positivity
claim can be completed by using Proposition~\ref{prop:SVc} which is stated later. 
\par
Finally we deal with the case of $\mathrm{rank}(\cM)=1$. In this case, the graph~$\Gamma$ associated with~$(X,\omega) \in \cM_1^{\ep, \thick}$ has vertices corresponding to the components of $X \setminus \cC$ and a horizontal edge for each cylinder. Note that all the cylinders are $\cM$-parallel and their heights are pairwise proportional. Hence, they form one $\cM$-cross-equivalence class, which implies that $\Gamma$ is indeed a principal graph in this case. The remaining part of the statement is obvious by the well-known degeneration behavior of Teichm\"uller curves.
\end{proof}
\par
The above proof indicates that we should decompose surfaces in $\cM_1^{\ep, \thick}$ further within their upper level and lower level in order to compute $c_\star(\calM,\Gamma)$ through the product decomposition. This will be achieved in Proposition~\ref{prop:almostbijection} later.  
\par 
Depending on the context, sometimes a collection of $\calM$-parallel cylinders is also called a ``configuration'' in the literature. In our context, these contain the partial information of the level graph after undegenerating the vertical level passage, retaining the horizontal edges only. This partial information is not sufficient for our aim, as is already visible by the fact that surfaces parametrized by such one-level horizontal boundary components have infinite volume, incompatible with the basic formula~\eqref{eq:cGamma}. 
\par

%%%%%%%%%%%%%%%%%%%%%%%%%%%%%%%%%%%%%%%%%%%%%%%%%%%%%%%%%%
\subsection{Principal boundary but higher codimension}
\label{sec:codim}
%%%%%%%%%%%%%%%%%%%%%%%%%%%%%%%%%%%%%%%%%%%%%%%%%%%%%%%%%

The goal of this section is to explain why principal graphs can define
boundary strata of codimension greater than one in the multi-scale compactification and yet give the \emph{principal} boundary for computing Siegel--Veech
constants for affine invariant submanifolds of rank greater than one. This
apparent contradiction can be clarified by considering the area renormalization.
\par
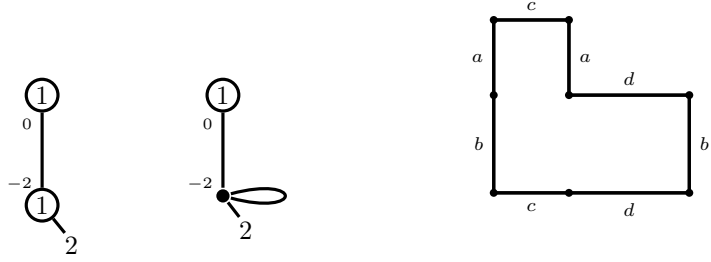
\begin{figure}
\begin{tikzpicture}[
		baseline={([yshift=-.5ex]current bounding box.center)},
		scale=2,very thick,
	%	node distance=\HoG, 
		bend angle=30,
		every loop/.style={very thick},
     		comp/.style={circle,fill,black,,inner sep=0pt,minimum size=5pt},
		order top left/.style={pos=\PotlI,left,font=\tiny},
		order top right/.style={pos=\PotrI,right,font=\tiny},		
		bottom right with distance/.style={below right,text height=10pt}]

%%%%%%%%%%%%%%%%%%%%%%%%%%
%%%%%%%%% LEVEL 1 %%%%%%%%%%%%
%%%%%%%%%%%%%%%%%%%%%%%%%%

  \begin{scope}[local bounding box = l]
\node[circled number] (T) [] {$1$}; 
%\node[comp] (B) [below=of T] {}
\node[circled number] (B) [below=of T] {$1$}
	edge 
	node [order bottom left] {$-2$}
	node [order top left] {$0$} (T);
\node [bottom right with distance] (B-2) at (B.south east) {$2$};
\path (B) edge [shorten >=4pt] (B-2.center);
\end{scope}
\begin{scope}  [shift={($(l.base) + (1.2*\hordist,0)$)}, local bounding box = ct]
\node[circled number] (T) [] {$1$}; 
\node[comp] (B) [below=of T] {}
	edge 
	node [order bottom left] {$-2$}
	node [order top left] {$0$} (T)
%	edge [bend left] 
%		node [order bottom left] {$-2$} 
%		node [order top left] {$0$} (T)
%	edge [bend right] 
%		node [order bottom right] {$-2$} 
%		node [order top right] {$0$} (T)
                edge [loop right] (B);
\node [bottom right with distance] (B-2) at (B.south east) {$2$};
\path (B) edge [shorten >=4pt] (B-2.center);
\end{scope}

\begin{scope} [scale=0.5, shift={($(l.base) + (6*\hordist,0)$)}, local bounding box = ct]
  \tikzstyle{every node}=[font=\scriptsize]
    % bild 1
    \draw(0,0) node(P1)[inner sep=0]{} -- (-0.0,-1.3) node(P2)[inner sep=0]{}  
    -- (1,-1.3) node(P3)[inner sep=0]{}  -- (2.6,-1.3) node(P4)[inner sep=0]{}
    -- (2.6,0) node(P5)[inner sep=0]{} -- (1,0) node(P6)[inner sep=0]{}   
    -- (1,1) node(P7)[inner sep=0]{} -- (0,1) node(P8)[inner sep=0]{} -- (P1);         
    \draw (P1)--(P2) node[left,midway] {$b$} 
    (P2)--(P3) node[below,midway] {$c$} 
    (P3)--(P4) node[below,midway] {$d$} 
    (P4)--(P5) node[right,midway] {$b$} 
    (P5)--(P6) node[above,midway] {$d$}                                
    (P6)--(P7) node[right,midway] {$a$} 
    (P7)--(P8) node[above,midway] {$c$} 
    (P8)--(P1) node[left,midway] {$a$}; 
    \fill (P1) circle (1.5pt) (P2) circle (1.5pt) (P3) circle (1.5pt)
    (P4) circle (1.5pt) (P5) circle (1.5pt) (P6) circle (1.5pt)
    (P7) circle (1.5pt) (P8) circle (1.5pt);	         
\end{scope}

%
%\begin{scope}[shift={($(l.center) + (\hordist,0)$)}, local bounding box = h]
%\node[circled number] (T) [] {$1$} 
%	edge [loop right] (T);
%\node [bottom right with distance] (T-2) at (T.south east) {$2$};
%\path (T) edge [shorten >=4pt] (T-2.center);
%\end{scope}
%
\end{tikzpicture}
\caption{The level graphs $\Gamma_1$ (left) and $\Gamma_2$ (middle) and parameters on a flat surface in $\Omega \calM_2(2)$ (right)} \label{cap:H2}
\end{figure}

We consider the stratum $\omoduli[2](2)$ as an example. Let $D_{\Gamma_1}$ be the boundary stratum of codimension one given by the level graph $\Gamma_1$ of compact type with a vertex of genus one at each of the two levels, and let~$D_{\Gamma_2}$ be the boundary stratum of codimension two obtained by degenerating the lower level vertex of~$\Gamma_1$ into a horizontal loop (see Figure~\ref{cap:H2}). The path~$\gamma_1(\ep)$ in $\omoduli[2](2)$ consisting of flat surfaces of area one as in this figure with $a=\ep$ and $c=\ep$ while $b=1$ and $d=1-\ep-\ep^2$ converges to a point in $D_{\Gamma_1}$ as $\ep$ tends to zero. On the other hand, the path~$\gamma_2(\ep)$ with $a=\sqrt{\ep}$ and $c=\ep$ while $b=1$ and $d=1-\ep-\ep^{3/2}$ converges to a point in $D_{\Gamma_2}$. (This can be rigorously justified in the quasi-conformal topology.
To construct the maps, say for~$\gamma_2(\ep)$, cut out a neighborhood of the small top cylinder and cut this cylinder moreover along its core curve. This sequence of subsurfaces rescaled by a factor of~$1/\sqrt{\ep}$ converges quasi-conformally to a four-punctured sphere on the bottom level, and the rest 
of the surface converges to a flat torus on the top level.)
\par
This implies that neither the limit point $\gamma_1(0)$ of the first path
nor any small neighborhood of it lies in a small tubular neighborhood of~$D_{\Gamma_2}$.
Nevertheless, any point in the path~$\gamma_1(\ep)$ for $\ep < \ep_0$
lies in the $\ep' = \ep_0/(1-\ep_0^2)$-neighborhood $\cM_{1,\Gamma_2}^{\ep'}$
around the configuration~$\Gamma_2$. In fact, the top level surface certifying
this for the point $\gamma_1(\ep)$ is the bottom cylinder rescaled
by $1/(1-\ep)$.
\par

%%%%%%%%%%%%%%%%%%%%%%%%%%%%%%%%%%%%%%%%%%%%%%%%%%%%%%%%%%
\subsection{Coordinates adapted to the principal boundary}
\label{sec:coordprinc}
%%%%%%%%%%%%%%%%%%%%%%%%%%%%%%%%%%%%%%%%%%%%%%%%%%%%%%%%%

For the remainder of this section, we make our standing hypothesis that~$\cM$ has REL zero. We fix a principal graph~$\Gamma$ of~$\cM$. We also fix a suitable basis of relative homology that allows for volume comparisons at the boundary and near the boundary using the normalization in Section~\ref{sec:VolNorm}. Let 
$$\rk \coloneqq \rank{\calM} \= \frac{m+1}{2} $$ 
and let $\cY_\Gamma^{\ep}$ be a component of $\cM_1^{\ep, \thick}$ with configuration~$\Gamma$.
\par
\begin{lemma} \label{le:adaptedbasis}
For any flat surface~$(X,\omega)$ near a point in $\cY_\Gamma^{\ep}$, there exists a subset $$\frakB \,:=\, \{\alpha_1,\ldots,\alpha_{\rk},\beta_1,
\ldots,\beta_{\rk}\} \,\subset\, H_1(X,Z(\omega), \bR)$$
such that the corresponding periods $a_i = \int_{\alpha_i} \omega$ and
$b_i = \int_{\beta_i} \omega$ are local coordinates of~$\calM$ with the
following properties:
\begin{itemize}
\item[i)] The periods $a_{1,j}$ of the core curves~$\alpha_{1,j}$ of the
$\calM$-parallel cylinders that correspond to the horizontal edges on lower level (called the \emph{horizontal core curves}) are multiples of the first basis
element $a_{1,j} = \fraka_j a_1$ for some constants $\fraka_j \in \bR$ and
all $j=1,\ldots, q$.
\item[ii)] There are relative cycles~$\beta_{1,j}$ joining a singularity on
either side of each horizontal cylinder (the \emph{horizontal crossing cycles})
such that $b_{1,j} := \int_{\beta_{1,j}} \omega = \frakb_j b_1$ for some
constants $\frakb_j \in \bR$ and  $j=1,\ldots, q$.
\item[iii)] The form~$\vartheta$ defined by~\eqref{eq:integrationcoordinates}
and~\eqref{eq:symptoimag} is the imaginary part of the area form~$h$.
\item[iv)] The set~$\frakB$ is symplectic, i.e., $\langle \alpha_i, \beta_j \rangle = \delta_{i,j}$ and $\{\alpha_2,\ldots,\alpha_{\rk},\beta_2, \ldots,\beta_{\rk}\}$ consists of absolute periods supported on the top level.
\end{itemize}
\end{lemma}
\par
\begin{proof} The first set of equations expresses the $\calM$-parallelity of
the cylinders and we can take~$\alpha_1$ to be an $\bR$-multiple of one of
these core curves. The second set of equations is a restatement of
the additional condition for principal graphs in the case of REL zero. For the third statement, 
we revisit the proof in \cite[Section~6.1]{NguVol} of the existence of a basis
$T_{(X,\omega)}\calM$
that satisfies iii). This merely requires that the extension of the symplectic
form to a hermitian form is diagonal of signature $(\rk,\rk)$ in the
basis~\eqref{eq:symptoimag}. We rescale the above~$\alpha_1$ and~$\beta_1$
so that the form is diagonal~$(1,-1)$ with respect to the corresponding $A_1$ and~$B_1$, 
run the Gram--Schmidt process to find the remaining~$A_i$ and~$B_i$, and
invert~\eqref{eq:symptoimag} to retrieve the desired~$\alpha_i$ and~$\beta_i$.
The last statement is a reformulation of the conclusion from
Lemma~\ref{le:rankdrop}.
\end{proof}
\par
The following lemma explains why one can simply describe the projection to the boundary for
affine invariant submanifolds of REL zero, neither requiring the ``slit construction'' and ``transporting
a hole'' as in \cite{emz} nor requiring modification differentials as in \cite{LMS}.
\par
\begin{lemma} \label{le:novertresidue}
Suppose that the affine invariant submanifold~$\cM$ has REL zero and~$\Gamma$ is
a principal graph. Then the residue of a multi-scale differential at a
vertical edge of any boundary point~$p = (Y, \bfeta)$ in $D_\Gamma \cap \ol{\cM}$ is zero.
\end{lemma}
\par
\begin{proof} To explain the basic idea, we fix a vertical edge~$e \in E(\Gamma)$ and suppose there exists a relative cycle~$\gamma$ on a flat surface~$(X,\omega)$ near the boundary point~$p$ such that $\gamma$ crosses (the annulus corresponding to)~$e$ but no other degenerating annulus. By Lemma~\ref{le:adaptedbasis}, the period of~$\gamma$ can be expressed as a linear combination of the periods of~$\alpha_1,\ldots,\beta_{\rk}$. By construction, these periods do not change along a loop around the boundary divisor corresponding to the vertical level passage. Since on the other hand the period of~$\gamma$ changes by $\ell/\kappa_e$ times the residue along such a loop, the residue must be zero.
\par
For the general case, consider any component~$v \in V(\Gamma)$ on the top level.
By the global residue condition, we have $\sum_{e \mapsto v} r_e = 0$, where the sum is
over all edges adjacent to~$v$, and where we take residues as integrals over
the core curves of the annuli connecting the top level to the bottom level oriented counterclockwise as boundary curves of the top level. Since any differential on any component of the bottom level has a zero, for any pair of edges $e_i, e_j \mapsto v$, there is a relative period crossing~$e_i$ and~$e_j$ but no other
horizontal nor vertical edges. The preceding argument and the orientation
convention imply that $\ell (r_{e_i}/\kappa_{e_i}  - r_{e_j}/\kappa_{e_j}) = 0$. 
Since $\kappa_{e_i} > 0$ for all vertical edges, together with the global residue condition this implies that $r_{e_i} = 0$ for all~$i$.
\end{proof}

%%%%%%%%%%%%%%%%%%%%%%%%%%%%%
\subsection{The parameter space for cylinders}
\label{subsec:cylinders}
%%%%%%%%%%%%%%%%%%%%%%%%%%%%%
We add more notations for the cylinders in a flat surface near the boundary of a principal
graph~$\Gamma$ that we fix throughout. For convenience we label the
cylinders by increasing order of periods of their waist
curves. Write $a_1 = we^{i\theta}$ and decompose $h = \Im(b_1 e^{-i\theta})$
and $t = \Re(b_1 e^{-i\theta})$ so that  the cylinders have   
$$ \text{waist length} \,\, w_j \= {\fraka_j} w\,, \quad
\text{height}\,\, h_j \= {\frakb_j} h\,,  \quad \text{and twist}\,\,
t_j \= {\frakb_j} t$$
respectively. The symplectic normalization of~$\frakB$ implies that
$$ A \,:=\, \sum_{j=1}^q \fraka_j\frakb_j = 1\,,$$
i.e., the total area covered by the cylinders is $A |a_1 \wedge b_1| = 
|a_1 \wedge b_1| =  w h$.
\par
The next result illustrates that surfaces near the principal boundary behave like Teichm\"uller curves. 
\par
\begin{lemma} \label{le:commensurable}
The ratios $\frakb_j/\fraka_j$ for $j=1,\ldots,q$ are commensurable. 
\end{lemma}
\par
\begin{proof} Consider the $j$-th shortest cylinders on two surfaces
in~$\cY_{1,\Gamma}^{\ep,\thick}$ with directions~$\theta$ resp.~$\theta'$, with
waist length $w_j$ resp.~$w_j'$ and crossing periods~$\frakb_j b_1$
resp.~$\frakb_j b_1'$. They are isomorphic as cylinders with the zeros on their boundary marked if and only if two conditions hold. First, if $w_j = w_j'$, hence $a_1 = a_1'$. Second, if $\frakb_j(b_1 -b_1') / \fraka_j a_1$ lies in a finite number of cosets of~$\bZ$, where the finite number stems from automorphisms permuting the zeros on their boundary. This means that the cylinders are isomorphic only if
\bes
\frac{b_1-b_1'}{a_1} \in \frac{\fraka_j}{R_j \frakb_j} \cdot \bZ
\qquad \text{for some $R_j \in \bN$}
\ees
that gives the order of cyclic permutations of the zeros at the boundaries.
\par
Suppose the claim is wrong. Then for $a_1$ in a small neighborhood of
an initial value~$a_{1,0}$ and all $b_1$ such that $b_1/a_1$  runs over
all $\bR$-translates of a neigborhood of an initial value $b_{1,0}/a_{1,0}$, 
it never happens that all the cylinders $\{C_j\}_{j=1}^q$ are isomorphic.
This contradicts the finiteness of the volume of the affine invariant submanifold as
we can vary the parameters of an initial top level surface in a small open
neighborhood (thus preserving the conditions of belonging to the thick part)
and independently vary~$a_1$ and~$b_1$ as prescribed. If two surfaces in this
set were isomorphic, the cylinders $\{C_j\}_{j=1}^q$ would have to be isomorphic,
as they are distinguished by the lengths of the core curves. Since this has been
ruled out, the $\bR$-translates of $b_1/a_1$ ensure that the volume is infinite.
\end{proof}
\par
In the sequel, we fix the local ``rotation number'' $R_j \in \bN$ of the preceding
proof to be maximal such that a $(1/R_j)$-th Dehn twist around its core curve
provides an automorphism of the $j$-th cylinder (with the zeros on its boundary
marked). Define
\be \label{eq:defcalLGamma}
\cL_\Gamma \coloneqq \lcm \Bigl\{\frac{\fraka_j}{R_j \frakb_j}, \, \, j=1,\ldots,q \Bigr\}\,.
\ee
\par
We let $\cP_{\ep}$ be the space of parameters of the $q$ cylinders
that are $\cM$-parallel, with relative sizes as above,
of total area one and such that at least one cylinder has a 
waist curve of size~$\leq \ep$ when rescaled by the total area. That is,
$$ \cP_{\ep}
\= \big\{(a_1,b_1) \in \CC^2\, : \,  |a_1 \wedge b_1| = 1,\,\,
\Re\Bigl(\frac{b_1}{a_1}\Bigr) \in [0,\cL_\Gamma], \,\,\fraka_1|a_1|
 \leq \ep\big\}\,.
$$ 
 It will also be useful to consider 
$$ \cP_{\ep, \leq 1}
\= \big\{(a_1,b_1) \in \CC^2\, : \,  |a_1 \wedge b_1| \leq 1,\,\,
\Re\Bigl(\frac{b_1}{a_1}\Bigr) \in [0,\cL_\Gamma], \,\,
\fraka_1|a_1| \leq \ep  \big\}\,.
$$
For a subset~$U$ in the hypersurface of flat surfaces of area one, we denote by $C_s(U)$ the cone
under~$U$ of area bounded by~$s$, i.e., 
\be
C_s(U) \= \{(X,\omega): \vol_\omega(X) \leq s, \,\, (X,\omega)/\sqrt{\vol_\omega(X)}
\in U\}\,
\ee
and we omit the subscript for $s=1$.
\par
Note that $\cP_{\ep, \leq 1}$ is different from $C(\cP_{\ep})$ and that it contains in particular all $r\cP_{\ep'/r}$ for $r\in (0,1]$ and $\ep'\in (0,\ep]$. Finally we let
\be
R \= \{(r_\top, r_C) \in (0,1)^2 \,:\, r_\top^2 + r_C^2 < 1 \}
\ee
be the space parameterizing the relative sizes of the top level and
the cylinders. 

%%%%%%%%%%%%%%%%%%%%%%%%%%%%%
\subsection{Components of the principal boundary}
\label{subsec:components}
%%%%%%%%%%%%%%%%%%%%%%%%%%%%%

Here we want to decompose surfaces in $\cM_1^{\ep, \thick}$ further in their upper level and lower level in order to compute $c_\star(\calM,\Gamma)$ via the product decomposition. Let $\cY_{1,\Gamma}^{\ep}$ be a connected component of $\cM_1^{\ep, \thick}$ with configuration~$\Gamma$.  Recall the maps $p_\Gamma$ and $c_\Gamma$ in Diagram~\eqref{dia:covboundary}. 
\par
\begin{definition}
\label{def:Y}
Define $\cY_\Gamma^{\top}$ and $\cY_\Gamma^{\bot}$ to be the linear submanifolds obtained by taking the Zariski closure of the top level projection and bottom level projection of $p_{\Gamma}(c_\Gamma^{-1}(\ol{\CC^* \cdot \cY_{1,\Gamma}^{\ep}} \cap D_\Gamma))$. Moreover, define $\cY_{1,\Gamma}^{\top}$ to be the intersection of~$\cY_\Gamma^{\top}$ with the (real) hypersurface of flat surfaces of area one (with our standing volume normalization thanks to Lemma~\ref{le:rankdrop}).
\end{definition}
\par
By definition of principal graphs, the  projectivized stratum $\PP \cY_\Gamma^{\bot}$ is a finite set. Also by the assumption that $\cY_{1,\Gamma}^{\ep}$ is connected, indeed $\PP \cY_\Gamma^{\bot}$ consists of a single point which parameterizes a unique bottom differential up to scaling.    
\par
\begin{prop} \label{prop:almostbijection} 
There exists a connected subset $\cF^\top \subset \cY_{1,\Gamma}^{\top}$ and a connected subset $\ol{\cF} \subset \cY_{1,\Gamma}^{\ep}$ with a finite covering $\cF \to \ol{\cF}$ of degree $\Aut_{\cY_\Gamma}(\Gamma)$ such that: 
\begin{itemize}
\item[(i)] The top level subset has almost full volume $\nu_{\cM_1^\top}(
\cY_{1,\Gamma}^{\top} \setminus \cF^\top) = o(\ep^2)$. 
\item[(ii)] The subset in the stratum has almost full volume
$\nu_{\cM_1}(\cY_{1,\Gamma}^{\ep} \setminus \ol{\cF}) = o(\ep^2)$. 
\item[(iii)] There is a map  $\sigma\colon C(\cF) \to C(\cF^\top) \times \cP_{{\ep},\leq 1}$  that preserves the periods $a_1,\ldots, b_{\rk}$. The image of $\sigma$ fibers over~$R$ under the map taking square root of the areas of both factors, and the
fiber over~$(r_\top,r_C)$ is ${r_\top}\cF^\top \times r_C \cP_{\ep\sqrt{r_\top^2+r_C^2}/r_C,}$. Moreover, the map~$\sigma$ is a covering over its image.
\end{itemize}
\end{prop}
\par 
We need to record the degree of these maps for later use:
\par
\begin{definition}
\label{def:K}
Given a connected principal boundary component $\cY_{1,\Gamma}^{\ep}$, define the number of \emph{reachable decorations} for $\cY_{1,\Gamma}^{\ep}$ to be
$\cK_{\cY_\Gamma} = \deg(\sigma)/|\Aut_{\cY_\Gamma}(\Gamma)|$.
\end{definition}
\par
Note that $\cK_{\cY_\Gamma}$ is well-defined, as the degree of $\sigma$ is constant on~$\cF$ since we have assumed that these full measure sets are connected. 
It comprises choices of reachable prong-matching equivalence classes (in the sense of \cite{chernlinear}), rescaled by the size of the graph automorphism subgroup preserving the linear submanifold, in fact
\be
\cK_{\cY_\Gamma} = \frac{K^\calM_\Gamma }{|\Aut_{\cY_\Gamma}(\Gamma)|\, \ell_\Gamma}
\ee
by~\eqref{eq:ratio_degrees}.
\par
\begin{proof}[Proof of Proposition~\ref{prop:almostbijection}] 
As a first approximation for the set~$\ol\cF$, we take $\cY_{1,\Gamma}^\ep$, which we recall to be defined as a component of the thick part in the $\ep$-neighborhood given by the bullet points in the proof of Proposition~\ref{prop:locSV}.
\par
To define the second factor of~$\sigma$, the map to $\cP_{{\ep}, \leq 1}$, choose
a basis of homology as in Lemma~\ref{le:adaptedbasis}. Note that $\alpha_1$ is
globally well-defined if we scale it so that $\fraka_1 = 1$ (and scale~$\beta_1$
by the inverse to preserve condition iv)). Moreover, $\beta_1$ is well-defined among
all choices satisfying condition~ii) up to adding $\cL_\Gamma \alpha_1$ by the
argument leading to the definition~\eqref{eq:defcalLGamma}. Thus~$\beta_1$
is unique with the requirement that $\Re(b_1/a_1) \in [0,\cL_\Gamma)$. We take $(a_1,b_1)$ with this choice of basis. To see that this defines indeed a map to $\cP_{{\ep}, \leq 1}$, observer that $|a_1 \wedge b_1|$ is now the total area of the space covered by cylinders, which is indeed less than one. The last condition in~$\cP_{\ep, \leq 1}$ reflects the length of the shortest cylinder waist curve.
\par
To define the first factor of~$\sigma$, we replace $\ol{\cF}$ by a connected and simply connected set of full measure and pass to a covering $\cF \to \ol\cF$ where all the edges of~$\Gamma$ are marked. This can be achieved by a map of degree $\Aut_\cY(\Gamma)$. We now define as approximation for~$\sigma$ a map to~$C(\cY_{1,\Gamma}^\top)$. Informally, the boundary components of the
tubular neighborhoods~$T$ from the proof of Proposition~\ref{prop:locSV} can be
embedded into small discs (since the residues are zero by Lemma~\ref{le:novertresidue}), so we fill in those discs. Formally, in order to also determine where to place the zero in the disc, we proceed similarly as in the proof of that lemma.
\par
Recall from \cite{BenBoundary} how the boundary of a linear submanifold is cut out locally, given the equations of the ambient linear submanifold in period coordinates. We only care about the top level and use the absence of horizontal nodes on top level to simplify the procedure. The recipe is to put equations in reduced row echelon form, suppress any appearance of lower level variables, and replace each cycle~$\gamma$ by its top level part~$\gamma^\top$ cutting it off at the nodes to lower level. By Lemma~\ref{le:adaptedbasis}, we can take the $\alpha_i$ and~$\beta_i$ there for $i \geq 2$ as pivot variables for the echelon form. This means that we can solve the equations uniquely for the missing absolute and relative periods in~$\cY_{1,\Gamma}^\top$. The last
condition in the definition of the thick part ensures that we can find a surface with the required periods by shrinking the short saddle connections. This procedure indeed does not depend on the chosen basis, since~$\alpha_1$ and~$\beta_1$ are suppressed in the equations and since a symplectic base change among the other variables can be accounted for by the corresponding base change on the top level surface.
\par
The claim that the image of~$\sigma$ fibers over~$R$ follows since the first factor of~$\sigma$ did not change the volume of the top level surface, because the volume is computed in terms of absolute periods. The claim that it fibers over ${r_\top}\cF^\top \times r_C \cP_{\ep\sqrt{r_\top^2+r_C^2}/r_C}$ is now obvious, since the shortest period of a core curve in a surface parameterized by $C(\cF)$ with image $(r_\top,r_C) \in R$ is bounded by~$\ep \sqrt{r_\top^2+r_C^2}$.
\par
We define $\cF^\top$ to be the thick part of $\cY_{1,\Gamma}^{\top}$ where the length
of the shortest saddle connection is at least~$4W\ep$ minus the boundary $\Re(b_1/a_i) =0$. This set certainly satisfies~(ii). We claim that the $\sigma$-image contains ${r_\top}\cF^\top \times r_C \cP_{\ep\sqrt{r_\top^2+r_C^2}/r_C,}$ for any $(r_\top,r_C) \in R$ and~$\ep$ small enough.
In fact, fix the reachable decorations (prong-matching and lower level surface)
that occur in the closure of $\cY^{\ep}_{1,\Gamma}$ and invert the procedure in
the definition of~$\sigma$: Replace discs around the zeros of the top level surface
by the lower level surface, with the infinite strips corresponding to horizontal edges replaced by cylinders according to the parameter in $\cP_{\ep\sqrt{r_\top^2+r_C^2}/r_C,}$.
Formally, again, we use the equations of the linear submanifold (and REL zero)
to revert the above procedure. This procedure changes periods by at most~$W\ep$
so that the surface we construct lies in  $\cY_{1,\Gamma}^{\top}$. We let~$\cF$ be the $\sigma$-preimage of~$\cF^\top$. It contains the thick part where the shortest saddle connection not at the boundary of any~$C_j$ has length at least
$5W\ep$ and thus satisfies~(i). The same argument as in the construction of preimages shows that~$\sigma$ is a covering over its image.
\end{proof}
\par

%%%%%%%%%%%%%%%%%%%%%%%%%%%%%%%%%%%%%%%%%%%%%%%%%%%%%%%%%%%%

%%%%%%%%%%%%%%%%%%%%%%%%%%%%%%%%%%%%%%%%%%%%%%%%%%%%%%%%%%%%

%%%%%%%%%%%%%%%%%%%%%%%%%%%%%%%%%%%%%%%%%%%%%%%%%%%%%%%%%%
\section{Disintegration of the volume forms}
\label{sec:disintegration}
%%%%%%%%%%%%%%%%%%%%%%%%%%%%%%%%%%%%%%%%%%%%%%%%%%%%%%%%%

We continue to assume that~$\calM$ has REL zero in this section. The goal of this section is to convert the ``almost-period preserving bijection'' from Proposition~\ref{prop:almostbijection} into an ``almost measure equality'' that we give in its first form in Corollary~\ref{cor:measurealmost}. As a second step, we perform integration with Siegel--Veech weights and obtain a formula for the Siegel--Veech constants in terms of boundary contributions. 
\par
Consider the Lebesque measure and its disintegration.  We define 
\be \label{eq:Pdis}
\nu_P \= \left(\frac{\bf i}{2}\right)^2da_1\,d\ol{a}_1\,db_1\,d\ol{b}_1
\= r^3\, dr\, d\nu^P_{\rm dis}
\ee
on the cylinder parameter space $\CC^*\cdot\cP_\ep \subset \CC^2\simeq \RR^4$. Obviously, the volume $\Vol(\cP_{\ep}) := \nu^P_{\rm dis}(\cP_{\ep})$ is finite. In fact, we have:
\par
\begin{lemma}\label{lem:volp} The volume of the parameter space $\cP_\ep$ is given by 
\be \label{eq:volPdef}
\Vol(\cP_\ep) \= \dim_\RR(\CC^*\cdot \cP_\ep)\int_{C(\cP_\ep)}\nu_P
\=\frac{\cL_\Gamma}{\fraka_1^2}\cdot 2\pi\ep^2\,.\ee
\end{lemma}
\par
\begin{proof} The first equality is another instance of the comparison in~\eqref{eq:volrelation}.
  \par
  Recall that the Lebesgue measure decomposes into $\nu_P=d\theta \, w\, dw \, dh \, dt$ where $a_1=we^{i\theta}, h=\Im(b_1e^{-i\theta})$, and $t=\Re(b_1e^{-i\theta})$. Hence, $|a_1\wedge b_1|=wh$, $\Re\left(\frac{b_1}{a_1}\right)=\frac{t}{w}$, and $C(\cP_{\ep})$ is then parameterized by
\[C(\cP_{\ep})=\left\{(\theta,w,h,t)\in[0,2\pi)\times \RR_+^3 : \, wh\leq 1, \frac{t}{w}\in[0, \cL_{\Gamma}], \frac{\fraka_1 w}{\sqrt{wh}}<\ep\right\}\,. \]
In this coordinate system, we compute the volume of the parameter space to be
\begin{align*} &\phantom{\=} \int_{C(\cP_\ep)}\nu_P
\=\int_0^{2\pi}d\theta\int_{\substack{wh\leq 1\\\fraka_1^2 w\leq \ep^2 h}}\left(\int_0^{\cL_\Gamma w} dt\right)w\,dw\,dh
\= 2\pi\cL_\Gamma\int_{\substack{wh\leq 1\\\fraka_1^2 w\leq \ep^2 h}}w^2\,dw\,dh\\
&\= 2\pi\cL_\Gamma\left(\int_0^{\frac{\fraka_1}{\ep}}\left(\int_{w\leq \frac{\ep^2h}{\fraka_1^2}}w^2dw\right)dh+\int_{\frac{\fraka_1}{\ep}}^\infty\left(\int_{w\leq \frac{1}{h}}w^2dw\right)dh\right) \=\frac{\pi}{2}\frac{\cL_\Gamma}{\fraka_1^2}\ep^2\,.
\end{align*}
\end{proof}
\par
Recall that $\rk=(m+1)/2$ is the rank of $\calM$. Then the parametrization of the components of $\calM_{1,\Gamma}^{\ep,\thick}$ implies:
\par
\begin{cor} \label{cor:measurealmost}
The measure of any component of the $\ep$-thick part with
configuration~$\Gamma$ can be estimated as 
\be
\Vol(\cY^\ep_{1,\Gamma}) \=
\frac{\dim_\RR(\cM)}{\vn{16}(m-1)(m+1)} \cdot \cK_{\cY_\Gamma} \cdot \Vol(\cY_{1,\Gamma}^{\top}) \cdot \Vol(\cP_{\ep})
+ o(\ep^2) 
\ee
as $\ep \to 0$.
\end{cor}
\par
\begin{proof}
Let $\vartheta$ be the $2$-form in~\eqref{eq:integrationcoordinates}
obtained from the coordinates in Lemma~\ref{le:adaptedbasis}. We define 
\be
\vartheta^\top\coloneqq \frac{{\bfi}}{2} \Big( \sum_{i=2}^{\rk} d{A_i}
\wedge d\bar{A_i} -  d{B_i} \wedge d\bar{B_i} \Big)\,
\ee
as the corresponding $2$-form omitting the first pair of variables that corresponds to the bottom level. Then a routine calculation shows that
\ba \label{eq:wedgeproduct}
\nu_{\rm int}=\frac{\vn{(\bfi\vartheta)}^{m+1}}{(m+1)!} \=
\left(\frac{\vn{\bfi^{3m+3}}}{2^{2m+2}}\right)
\prod_{i=1}^{\rk} da_i \wedge d\ol{a}_i \wedge  db_i\wedge d\ol{b}_i\,, \\
\nu^\top_{\rm int}=\frac{\vn{(\bfi\vartheta\top)}^{m-1}}{(m-1)!} \=
\left(\frac{\vn{\bfi^{3m-3}}}{2^{2m-2}}\right)
\prod_{i=2}^{\rk} da_i \wedge d\ol{a}_i \wedge  db_i\wedge d\ol{b}_i\,.
\ea
Recall from Lemma~\ref{le:adaptedbasis}~iii) and a similar volume normalization on the top level. We have
\be \label{eq:twomeasures}
\nu_{\calM_1}(B) \= \int_{C(B)} \frac{\vn{(\bfi\vartheta)}^{m+1}}{(m+1)!}\,,
\qquad
\nu_{\calM_{1,\Gamma}^\top}(B') \= \int_{C(B')} {\frac{(\vn{\bfi}\vartheta^{\top})^{m-1}}{(m-1)!}\,}
\ee
for measurable sets $B \subset \calM_1$  and $B' \subset \calM_{1,\Gamma}^\top$
contained in any chart, where the expressions for the differential forms are
defined and where $C(\cdot)$ denotes the cones under the hypersurface of area-one flat surfaces 
in both cases. We apply this to $\cY^\ep_{1,\Gamma}$ and $\cY_{1,\Gamma}^{\top}$
(covering the whole space by appropriate charts and adding contributions)
and the map of cones~$\sigma$ from Proposition~\ref{prop:almostbijection} on the almost full measure sets defined there. We then obtain that
\ba\label{eq:YepVol} 
\Vol(\cY^\ep_{1,\Gamma}) &\stackrel{\eqref{eq:defmeasure}, \eqref{eq:volrelation}, \eqref{eq:wedgeproduct}} {{\=}} \dim_\RR(\cM) \int_{C(\cY^\ep_{1,\Gamma})}
\left(\frac{\vn{\bfi^{3m+3}}}{2^{2m+2}}\right)
da_1\cdots \,d\ol{b}_{\rk} \\
&\=  \cK_{\cY_\Gamma} \dim_\RR(\cM)  \int_{\Im(\sigma)}
\left(\frac{\vn{\bfi^{3m+3}}}{2^{2m+2}}\right)
da_1\cdots\, d\ol{b}_{\rk} + o(\ep^2) \\
& \stackrel{\eqref{eq:wedgeproduct}}{\=} \cK_{\cY_\Gamma} \dim_\RR(\cM) \frac{\bfi^{\vn{6-2}}}{2^2}  \int_{\Im(\sigma)}  d\nu_P d\nu_{\rm int}^\top + o(\ep^2) \\
&\stackrel{\eqref{eq:defnudis},\eqref{eq:Pdis}} {{\=}} 
\cK_{\cY_{\Gamma}} \vn{\frac{\dim_\RR(\cM)}{2^2}}  \int_R  \Bigl(r_\top^{2m-3} \int_{\cY^\top_{1,\Gamma}}d\nu_{\rm dis}\,\cdot \\
& \qquad \qquad \qquad\qquad\qquad\
r_C^3 \int_{\cP_{\ep\sqrt{r_\top^2+r_C^2}/r_C}}\!d\nu^P_{\rm dis}\Bigr) \,\,\, dr_\top d{r_C}+ o(\ep^2) \\
&\stackrel{\eqref{eq:volrelation},\eqref{eq:volPdef}} {{\=}}
\cK_{\cY_\Gamma} 
\vn{\frac{\dim_\RR(\cM)}{2^2}}  \Vol(\cY^\top_{1,\Gamma})\Vol(\cP_{\ep}) \cdot \\
& \qquad \qquad \qquad\qquad\qquad
\int_R\left(\frac{r_\top^2+r_C^2}{r_C^2}\right) r_\top^{2m-3} r_C^3 \, dr_\top d{r_C}+ o(\ep^2) \\
&\= \cK_{\cY_\Gamma}\frac{\dim_\RR(\cM)}{\vn{16}(m-1)(m+1)}
\Vol(\cY^\top_{1,\Gamma})\Vol(\cP_{\ep})+ o(\ep^2)\,, 
\ea
where we used the equality
\be \label{eq:Gammaintegral}
\int_R r_\top^{2a-1}r_C^{2b-1}dr_\top dr_C \= 
\frac{1}{4}\frac{(a-1)!(b-1)!}{(a+b)!}\,
\ee
for the last line and to derive the fourth line we used that  $\Vol(\cP_{c\cdot\ep})=c^2\cdot\Vol(\cP_{\ep})$ by Lemma~\ref{lem:volp}, with specifically $c=\sqrt{r_\top^2+r_C^2}/r_C$.
\end{proof}
\par
To compute area or cylinder Siegel--Veech constants, note that they only depend
on short closed geodesics in cylinders and are thus controlled by the point
in the parameter space~$\cP_\ep$. Therefore, we introduce the following weight
functions. Recall that $a_1 = we^{i\theta}$, let $h_i = \Im(b_ie^{-i\theta})$, 
and note that $ \sum_{j=1}^q \fraka_j h_i =h$. The volume of the cylinder
space can be written equivalently as $|a_1 \wedge b_1| = wh$. Letting $\fraka_{q+1}:=\infty$, we define
\be\begin{aligned}
W^{\ep}_{\star}(w,h)
&\=\sum_{i=1}^q\chi(w\fraka_i\leq \ep\sqrt{wh} \leq
w\fraka_{i+1})f_{\star}(i)\\
&\= \sum_{i=1}^q\chi\left(\frac{w\fraka_i^2}{\ep^2}\leq h\leq
\frac{w\fraka_{i+1}^2}{\ep^2}\right)f_{\star}(i)\,  \label{eq:weights}
\end{aligned}
\ee
for $\star \in \{\area,\cyl\}$, where 
\[f_{\area}(i) \= \sum_{j=1}^i \fraka_j\frakb_j\,, \quad f_{\cyl}(i) \=i\,. \] 
Note that $ f_{\area}(q)=1$.
\par
\begin{lemma} \label{le:SVparameterspace}
  The Siegel--Veech weighted volume of the parameter space
associated with~$\Gamma$ for $\star \in \{\area, \cyl\}$ is
\be \label{eq:Volstar}
\Vol_\star(\cP_\ep) \coloneqq \dim_\RR(\CC^*\cdot\cP_\ep)\int_{C(\cP_\ep)} W^\ep_{\star}(w,h) d\nu_P \=2\pi\ep^2c_{\Gamma}^{\star}
\ee
with
\[c_{\Gamma}^{\area} \= {\cL_\Gamma}\sum_{j=1}^q\frac{\frakb_j}{\fraka_j}\,,
\qquad  c_{\Gamma}^{\cyl} \=\cL_\Gamma\sum_{j=1}^q \frac{1}{\fraka_j^2}\,.\]
\end{lemma}
\par
\begin{proof}
Since the weights neither depend on the twist parameters~$t_i$ nor on the
argument of~$a_i$, and since all functions involved are positive, we can integrate as in Lemma~\ref{lem:volp}:
\bas
\int_{C(\cP_\ep)} W^\ep_{\star}(w,h) d\nu_P & \=
\int_0^{2\pi}d\theta\int_{\substack{wh\leq 1\\ \fraka_1^2 w\leq  \ep^2 h}}
\left( \int_0^{ \cL_\Gamma w}dt\right)W^{\ep}_{\star}(w,h)w\,dw\,dh\\
&\= 2\pi \cL_\Gamma \int_{\substack{wh\leq 1\\  \fraka_1^2w\leq \ep^2 h}} W^{\ep}_{\star}(w,h)w^2\,dw\,dh\,.
\eas
\par
This last integral is the sum over $i\in\{1, \dots, q\}$ of the following
integrals
\begin{align*}
I^{\star}_i&\=f_{\star}(i)\int_{\substack{wh\leq 1\\ \fraka_1^2 w\leq \ep^2 h}}w^{2}\chi\left(\frac{w\fraka_i^2}{\ep^2}
\leq h\leq \frac{w\fraka_{i+1}^2}{\ep^2}\right)\,dw\,dh\\
&\= f_{\star}(i)\int_{w=0}^{\frac{\ep}{\fraka_{i+1}}} w^{2} \left(\int_{0}^{\infty}
\chi\left(\frac{w\fraka_i^2}{\ep^2}\leq h\leq \frac{w\fraka_{i+1}^2}{\ep^2}\right)
\,dh\right)\,dw\\
&\quad + f_{\star}(i)\int_{w=\frac{\ep}{\fraka_{i+1}}}^{\frac{\ep}{\fraka_{i}}}
w^{2}
\left(\int_{0}^{\infty}\chi\left(\frac{w\fraka_i^2}{\ep^2}\leq h\leq
\frac{1}{w}\right)\,dh\right)\,dw\\
&\= f_{\star}(i)\int_{0}^{\frac{\ep}{\fraka_{i+1}}} w^{2} 
\left(\frac{w\fraka_{i+1}^2}{\varepsilon^2}-\frac{w\fraka_{i}^2}
{\varepsilon^2}\right)\,dw+ f_{\star}(i)\int_{\frac{\ep}{\fraka_{i+1}}}^{\frac{\ep}{\fraka_{i}}}
w^{2} \left(\frac{1}{w}-\frac{w\fraka_{i}^2}{\varepsilon^2}\right)\,dw\\
&\=f_{\star}(i)\frac{\varepsilon^2}{4}
\left(\frac{1}{\fraka_i^2}-\frac{1}{\fraka_{i+1}^2}\right)\,,
\end{align*}
where the first decomposition comes from the domain integration illustrated in Figure~\ref{fig:domain_int}. Combining the two computations gives the claim.
\end{proof}
\par
\begin{figure}
\begin{tikzpicture} 
\clip(-0.5,-1) rectangle (5.5,3.5);
\fill[line width=1pt,color=sqsqsq,fill=sqsqsq,fill opacity=0.1] (0.8074165757338932,1.2385180463889536) -- (0.8406795169162544,1.1895139347133832) -- (0.8687199082765512,1.1511190090991408) -- (0.9023683779089073,1.1081948619668367) -- (0.9360168475412635,1.068356838476581) -- (0.9724693563096493,1.0283100372384224) -- (1,1) -- (1.0313541781662723,0.9695990195899342) -- (1.0762188043427472,0.9291790813957256) -- (1.112671313111133,0.8987380084455548) -- (1.1491238218795188,0.8702282390807892) -- (1.1883803697839341,0.8414814191030565) -- (1.2,0.8333333333333334) -- (1.250069230776587,0.7999556947568134) -- (1.2917732636099657,0.7741296620472067) -- (0,0) -- cycle;
\draw[line width=1pt,smooth,samples=100,domain=-0.34497890658688746:4.224854033173722] plot(\x,{1/(\x))});
\draw[line width=1pt,smooth,samples=100,domain=-0.34497890658688746:4.224854033173722] plot(\x,0);
\draw [line width=1pt,domain=-0.34497890658688746:4.224854033173722] plot(\x,{(-0--1.2385180463889536*\x)/0.8074165757338932});
\draw [line width=1pt,domain=-0.34497890658688746:4.224854033173722] plot(\x,{(-0--0.7741296620472067*\x)/1.2917732636099657});
\draw [line width=1pt,domain=-0.34497890658688746:4.224854033173722] plot(\x,{(-0--0.33285400976639307*\x)/3.004320124314651});
\draw [line width=1pt,dash pattern=on 1pt off 2pt] (0.8074165757338932,1.2385180463889536)-- (0.7974793283532654,0);
\draw [line width=1pt,dash pattern=on 1pt off 2pt] (1.2917732636099657,0.7741296620472067)-- (1.2945004335840267,0);
\draw (-0.2621420557150939,3.3977869731472556) node[anchor=north ] {$H$};
\draw (3.941828126028428,-0.1) node[anchor=north ] {$w$};
\draw (1.1736966927293273,2.8662505133865817) node[anchor=north ] {$\cfrac{\fraka_{i+1}^2w}{\varepsilon^2}$};
\draw (2.8,2.45) node[anchor=north] {$\cfrac{\fraka_{i}^2w}{\varepsilon^2}$};
\draw (3.831378991532703,1.2) node[anchor=north] {$\cfrac{\fraka_1 w}{\varepsilon^2}$};
\draw (0.5,3.5) node[anchor=north] {$\cfrac{1}{w}$};
\draw (0.6421602329686521,0.070506796463558) node[anchor=north] {$\cfrac{\varepsilon}{\fraka_{i+1}}$};
\draw (1.215115118165224,0.08431293827552355) node[anchor=north ] {$\cfrac{\varepsilon}{\fraka_i}$};
\end{tikzpicture}
\caption{Integration domain in the variables $w$ and $H$}\label{fig:domain_int}
\end{figure}
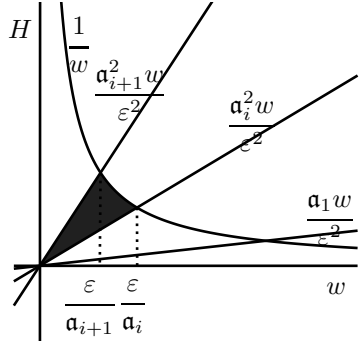

Now we can explicitly compute the local Siegel--Veech constants $c_\star(\cM,\Gamma)$ in Proposition~\ref{prop:locSV} by adapting the computations of Corollary~\ref{cor:measurealmost} with the Siegel--Veech weight.
\par
\begin{prop}\label{prop:SVc}
The local Siegel--Veech constants are given by
\[c_\star(\cM,\Gamma) \= \sum_{\cY_{1,\Gamma}^{\ep}}  c_\star(\cM,\cY_{1,\Gamma}^{\ep})\]
with \[c_{\cyl}(\cM,\cY_{1,\Gamma}^{\ep})
\=\frac{\dim_\RR(\cM)}{\vn{8}(m-1)(m+1)} 
\cdot \cK_{\cY_\Gamma}\cdot c_\Gamma^{\cyl}\cdot \frac{\Vol(\cY_{1, \Gamma}^\top)}{\Vol(\cM)}\] and
\[c_{\area}(\cM,\cY_{1,\Gamma}^{\ep})= \frac{\dim_\RR(\cM)}{\vn{8}(m-1)m(m+1)}
\cK_{\cY_\Gamma} \cdot c_\Gamma^{\area}
\cdot \frac{\Vol(\cY_{1, \Gamma}^\top)}{\Vol(\cM)}\]
where $c_\Gamma^{\area}$ and $c_\Gamma^{\cyl}$ are defined in Lemma~\ref{le:SVparameterspace}.
\end{prop}
\par
We remark that in the case of $\cM$ being the minimal strata of Abelian differentials, there is only one cylinder pinched in the principal boundary, i.e., $q = 1$. In this case, due to the normalization choice of $\fraka_1 = \frakb_1 = 1$, the above formulas imply immediately $c_\Gamma^{\rm area} = c_\Gamma^{\rm cyl}$ for all $\Gamma$, and consequently $c_{\rm area}(\cM) = \dim_{\bC} (\cM)\cdot c_{\rm cyl}(\cM)$ for the minimal strata. Indeed, this relation holds for all strata of Abelian differentials, originally due to Vorobets \cite[Theorem~1.6]{vorobets} and further explained to us by Zorich. The relation also holds for strata of quadratic differentials due to \cite{GoujSV}. However, this relation does not hold for a number of special affine invariant submanifolds, e.g., for many Teichm\"uller curves, starting already with the $L$-shaped surface consisting of three squares. It would be interesting to characterize affine invariant submanifolds $\cM$ for which $c_{\rm area}(\cM) = \dim_{\bC} (\cM)\cdot c_{\rm cyl}(\cM)$ holds.
\par
\begin{proof} We use the definition of $c_\star(\calM,\Gamma)$ given in~\eqref{eq:cGamma} and decompose $\Vol_{\star}(\cM_{1,\Gamma}^\ep)$ into the sum of the volumes of the connected components $\cY_{1, \Gamma}^\ep$, which gives the first equation. We then compute the weighted volumes 
  \[\Vol_\star(\cY_{1, \Gamma}^\ep) \=
  \dim_\RR(\cM)\cdot \int_{\cY_{1, \Gamma}^\ep}W_\star(S)\,d\nu_\cM(S)\]
using Proposition~\ref{prop:almostbijection} as in Corollary~\ref{cor:measurealmost}. The weights $W_\star(S)$ on the fiber ${r_\top}\cF^\top \times r_C \cP_{\ep'}$ with $\ep'=\ep\sqrt{r_\top^2+r_C^2}/r_C$ are given by 
\[W_{\area}(S)\= W_{\area}^{\ep'}(w,h)\cdot \frac{r_\top^2}{r_\top^2+r_C^2}\quad \mbox{and} \quad W_{\cyl}(S) \= W_{\cyl}^{\ep'}(w,h)\,,\]
with $W_\star^{\ep}(w,h)$ given by \eqref{eq:weights}. For the area weight, we
then obtain as in~\eqref{eq:YepVol} that
\bas
\Vol_{\area}(\cY_{1, \Gamma}^\ep) &{\=} \dim_\RR(\cM) \int_{C(\cY^\ep_{1,\Gamma})} W_{\area}(S)\, \left(\frac{\vn{\bfi^{3m+3}}}{2^{2m+2}}\right) da_1\cdots \,d\ol{b}_{\rk} \\
&\=  \cK_{\cY_\Gamma} \dim_\RR(\cM) \left(\frac{\vn{1}}{2}\right)^2 \Vol(\cY^\top_{1,\Gamma}) \,
\int_R\Big(\Vol_{\area}(\cP_{\ep\sqrt{r_\top^2+r_C^2}/r_C}) \, \cdot \\
& \qquad\qquad\qquad\qquad\qquad \cdot \frac{r_{\top}^2}{r_{\top}^2+r_C^2}
r_\top^{2m-1} r_C^3 \Big)\, dr_\top\, d{r_C} +o(\ep^2)\\
&\= \cK_{\cY_\Gamma} \dim_\RR(\cM) \left(\frac{\vn{1}}{2}\right)^2  \Vol(\cY^\top_{1,\Gamma})\Vol_{\area}(\cP_{\ep})\int_R  r_\top^{2m-1} r_C^3 \, dr_\top\, d{r_C} +o(\ep^2) \\
&\= \frac{\cK_{\cY_\Gamma} \dim_\RR(\cM)}{\vn{16}(m-1)m(m+1)}
\Vol(\cY^\top_{1,\Gamma})\Vol_{\area}(\cP_{\ep}) +o(\ep^2)\,
\eas
where the last equality follows from~\eqref{eq:Gammaintegral}.
A similar computation shows that the integral with~$W_{\cyl}(S)$ produces the
same expression without ${r_{\top}^2}/(r_{\top}^2+r_C^2)$. As in~\eqref{eq:YepVol}
the $R$-integral gives $1/4(m-1)(m+1)$ instead of $1/{4(m-1)m(m+1)}$.
The claim now follows by inserting~\eqref{eq:Volstar} and
the fundamental limit formula~\eqref{eq:cGamma}.
\end{proof}

%%%%%%%%%%%%%%%%%%%%%%%%%%%%%%%%%%%%%%%%%%%%%%%%%%%%%%%%%%%%

%%%%%%%%%%%%%%%%%%%%%%%%%%%%%%%%%%%%%%%%%%%%%%%%%%%%%%%%%%%%

%%%%%%%%%%%%%%%%%%%
\section{Siegel--Veech constants as intersection numbers}
\label{sec:intersection}
%%%%%%%%%%%%%%%%%%%

In this section, we will show that the right-hand side of Formula~\eqref{eq:c-rel0} matches the flat geometric contribution to $c_{\area}(\calM)$ from the principal boundary strata of $\calM$ as described in Propositions~\ref{prop:locSV} and~\ref{prop:SVc}.  Recall that $m = \dim \bP \cM$. 
The case of $m=1$ corresponds to Teichm\"uller curves for which the conjectural formula has been known to hold. From now on we assume that $m\geq 2$ and we still assume that~$\calM$ has zero REL zero.
\par 
We first reduce the intersection calculation in the numerator of the right-hand side of Formula~\eqref{eq:c-rel0} to the boundary divisor with horizontal nodes. Let $\delta$ be the total boundary divisor of $\PP\ol{\cM}$ and $\delta_H$ the boundary divisor whose elements have horizontal nodes.  
\par
 \begin{lemma}
 \label{le:non-horizontal}
 Let $\cM$ be an affine invariant submanifold of REL zero.  Denote by $\delta_\Lambda$ a boundary divisor of $\PP\ol{\cM}$ with a two-level
 graph $\Lambda$ whose generic elements have no horizontal nodes. Then 
  $$\int_{\PP\ol{\cM}} \delta_\Lambda \xi^{m-1} = 0 \quad {\rm and} \quad \int_{\PP\ol{\cM}}  \xi^{m-1} \delta = \int_{\PP\ol{\cM}}  \xi^{m-1} \delta_H\,. $$ 
 \end{lemma}
 \par
 \begin{proof}
 Let $\cM_\Lambda^{\top}$ be the unprojectivized top level stratum of $\delta_\Lambda$.  We know that 
 $$ 2~{\rm rank}~\cM_\Lambda^{\top} \leq \dim  \cM_\Lambda^{\top} < \dim \cM = m+1 = 2\rk $$
 where we recall that $\rk$ denotes the rank of $\cM$. It follows that ${\rm rank}~\cM_\Lambda^{\top}\leq \rk - 1$. Therefore, the subspace spanned by the absolute periods of $\cM_\Lambda^{\top}$ has dimension at most $m-1$, and hence after projectivization, $\int_{\PP\ol{\cM}} \delta_\Lambda \xi^{m-1} = 0$. The other identity then follows right away.  
 \end{proof}
 \par 
Next, we want to apply the divisor class relation \cite[Proposition 4.1]{chernlinear} by taking one of the simple polar edges in each component of $\delta_H$ to go down to the lower level of some resulting two-level graphs $\Gamma$. Note that this forces all horizontal edges that are $\cM$-parallel to each other to go down simultaneously to the lower level of $\Gamma$. We claim that 
\begin{eqnarray}
\int_{\PP\ol{\cM}}  \xi^{m-1} \delta_H 
& = &  \sum_{\cY_{\Gamma}}  \cK_{\cY_\Gamma} \cdot c^{\rm area}_\Gamma \int_{\PP\cY_\Gamma^{\top}} \xi^{m-2}  
\end{eqnarray}
where the summation runs over all irreducible components $\cY_{\Gamma}$ of the codimension-two boundary strata associated to the graphs $\Gamma$, the factor $\cK_{\cY_\Gamma}$ is given in Definition~\ref{def:K}, and the factor $c_\Gamma^{\rm area}$ is defined in Lemma~\ref{le:SVparameterspace}. 
\par 
The presence of the factor $\cK_{\cY_\Gamma}$ follows from the same reason as in Corollary~\ref{cor:measurealmost}.  To justify the presence of the factor $c_\Gamma^{\rm area}$, consider a loop in $\cY_\Gamma$ once around the intersection with the horizontal boundary divisor. Computing the number of times this twists around each branch of~$\delta_H$ that passes through such a boundary point and adding up this contribution computes the intersection multiplicity. Recall the coordinate system near the boundary from Section~\ref{sec:coordprinc} and in particular the coordinates for the cylinders from Section~\ref{subsec:cylinders}. Fixing~$a_i$ and $\Im(b_1/a_1)$ while letting $\Re(b_1/a_1)$ run along $[0,\cL_\Gamma]$ is actually a closed loop and once around the boundary by the choice of~$\cL_\Gamma$ as~$\lcm$ in~\eqref{eq:defcalLGamma}. This loop induces $\cL_\Gamma / (\fraka_j/\frakb_j)$ times a $(1/R_j)$-fractional Dehn twist around the $j$-th cylinder. Since each such fractional Dehn twist contributes one to the intersection number, the claim follows. 
\par
Note that $\int_{\PP\cY^{\top}_\Gamma} \xi^{m-2} \neq 0$ only if the projectivized dimension of the bottom stratum $\bP\cY_\Gamma$ is zero. In this case, $\Gamma$ is precisely a principal graph as defined in Section~\ref{sec:principalgraphs} where the bottom stratum $\bP\cY^{\bot}_\Gamma$ consists of a single point due to the irreducibility assumption of $\cY_\Gamma$. This explains the notation $\cY_\Gamma$ which corresponds to $\cY_\Gamma^{\ep}$ defined in Definition~\ref{def:Y} as $\ep$ tends to zero. 
\par 
Combining the above discussions, we conclude that the intersection contribution of each component $\cY_{\Gamma}$ of the principal boundary to $c_{\rm area}(\cM)$ is given by 
\begin{eqnarray}
\label{eq:one-loop}
c_{\rm area}^{\mathrm{int}}(\cM,\Gamma) \,:=\, -\frac{1}{4\pi^2} \cdot c_\Gamma^{\rm area} \cdot \cK_{\cY_\Gamma}\cdot \frac{ \int_{\PP\cY_\Gamma^{\top}} \xi^{m-2}}{\int_{\PP\ol{\cM}} \xi^m}\,.
\end{eqnarray}
\par
Now we are ready to prove the main result. 
\par 
\begin{proof}[Proof of Theorem~\ref{thm:main}]
Comparing Proposition~\ref{prop:SVc} to \eqref{eq:one-loop}, it suffices to check that 
\be \label{eq:matching}
\frac{\dim_{\bR}(\cM)}{2(m-1)m(m+1)} \cdot \left(\frac{\bf i}{2}\right)^2 \cdot  \frac{\Vol (\cY_{1,\Gamma}^{\top})}{\Vol (\cM)} = - \frac{1}{4\pi^2} \frac{ \int_{\bP\cY_\Gamma^{\top}}  \xi^{m-2} }{ \int_{\PP\ol{\cM}}  \xi^{m}}\,. 
 \ee
Recall from \eqref{eq:vol-int} that 
$$\Vol (\cM) = \frac{2}{m!} \cdot (\vn{\bfi}\pi)^{m+1} \cdot \int_{\bP\ol{\cM}} \xi^{m}\, $$
and similarly 
 $$ \Vol (\cY_{1,\Gamma}^{\top}) = \frac{2}{(m-2)!} \cdot (\vn{\bfi}\pi)^{m-1}\cdot \int_{\bP\cY^{\top}_\Gamma} \xi^{m-2}\,. $$
Then \eqref{eq:matching} obviously holds.  
\end{proof}

%%%%%%%%%%%%%%%%%%%%%%%%%%%%%%%%%%%%%%%%%%%%%%%%%%%%%%%%%%%%

%%%%%%%%%%%%%%%%%%%%%%%%%%%%%%%%%%%%%%%%%%%%%%%%%%%%%%%%%%%%

\newlength\innerhordist
\setlength{\innerhordist}{.4cm}
\newlength\innervertdist
\setlength{\innervertdist}{.8cm}
\newlength\hornodeshift
\setlength{\hornodeshift}{.6cm}
\newlength\labeldist
\setlength{\labeldist}{.4cm}
\newlength\arrowshift
\setlength{\arrowshift}{.4cm}
\newif\iflabels

%%%%%%%%%%%%%%%%%%%%%%%%%%%%%%%%%%%%%%%%%%%%%%%%%%%%%%%%%%
\section{Examples of principal graphs}
\label{sec:example}
%%%%%%%%%%%%%%%%%%%%%%%%%%%%%%%%%%%%%%%%%%%%%%%%%%%%%%%%%

In this section, we describe the principal boundary for several examples of affine invariant submanifolds of REL zero. 

%%%%%%%%%%%%%%%%%%%%%%%%%%%%
\subsection{The minimal strata of Abelian differentials}
%%%%%%%%%%%%%%%%%%%%%%%%%%%%

The principal boundary for strata of Abelian differentials was already known in \cite{emz} (see also \cite{chenPB} for the description in terms of twisted differentials). In the case of the minimal stratum in genus $g$ with a unique zero, since every bottom vertex must contain at least one zero, the bottom of a principal graph $\Gamma$ consists of a unique vertex $(X_0, \omega_0)$ which admits a horizontal loop and connects to the top vertices $X_1, \ldots, X_k$ via $k$ separating edges. In particular, $(X_0, \omega_0)$ belongs to the generalized stratum in genus zero with signature $(2g-2, -1, -1, -2g_1, \ldots, -2g_k)^{\frakR}$, where the residue condition $\frakR$ requires the residue of $\omega_0$ to be zero at each vertical polar edge. Moreover, each top vertex $(X_i, \omega_i)$ belongs to a minimal stratum with signature $(2g_i -2)$,  where $g_1 + \cdots + g_k = g-1$. See Figure~\ref{fig:minimal} for an illustration of such a principal graph. 
\begin{figure}[htb]
%\bas \  
\begin{tikzpicture}[
		baseline={([yshift=-.5ex]current bounding box.center)},
		scale=2,very thick,
		bend angle=20,
		every loop/.style={very thick},
		comp/.style={circle,fill,black,,inner sep=0pt,minimum size=5pt},
		order bottom left/.style={pos=.25,left=-0.1,font=\tiny},
		order top left/.style={pos=.75,left=-0.1,font=\tiny},
		order bottom right/.style={pos=.25,right=-0.1,font=\tiny},
		order top right/.style={pos=.75,right=-0.1,font=\tiny},
		circled number/.style={circle, draw, inner sep=0pt, minimum size=12pt},
		bottom right with distance/.style={below right,text height=10pt},
		bottom mid with distance/.style={below,text height=10pt},
		bottom left with distance/.style={below left,text height=10pt},
		top right with distance/.style={above right,text height=10pt},
		top mid with distance/.style={above,text height=10pt},
		top left with distance/.style={above left,text height=10pt},
		right with distance/.style={right,text height=10pt},
		left with distance/.style={left,text height=10pt},
		ram bottom right with distance/.style={below right,font=\tiny},
		ram bottom mid with distance/.style={below,font=\tiny},
		ram bottom left with distance/.style={below left,font=\tiny},
		ram top right with distance/.style={above right,font=\tiny},
		ram top mid with distance/.style={above,font=\tiny},
		ram top left with distance/.style={above left,font=\tiny},
		ram right with distance/.style={right,font=\tiny},
		ram left with distance/.style={left,font=\tiny}]
\begin{scope}[local bounding box = a]
	\node[circled number] (T1) [xshift=-3\hornodeshift] {$g_1$};
	\node[circled number] (T2) [xshift=-1\hornodeshift] {$g_2$};
	\node[circled number] (T3) [xshift=4\hornodeshift] {$g_k$};
	\node[comp] (B) [below=of T1, xshift=2\hornodeshift] {}
		edge
			node [order bottom left] {$\shortminus2g_1$} 
			node [order top left] {$2g_1-2$} (T1)
		edge
			node [order bottom right] {$\shortminus2g_2$} 
			node [order top right] {$2g_2-2$} (T2)
                edge
			node [order bottom right] {$\,\,\,\shortminus2g_k$} 
			node [order top right] {$\,\,\,2g_k-2$} (T3)
                edge [loop right] (B);
		\node [bottom left with distance] (B-0-1) at (B.south west) {$2g-2$};
	\path (B) edge [shorten >=6pt] (B-0-1.center);
\end{scope}
\end{tikzpicture}
\caption{The principal graph of the minimal stratum $\cH(2g-2)$}
	\label{fig:minimal}
\end{figure}
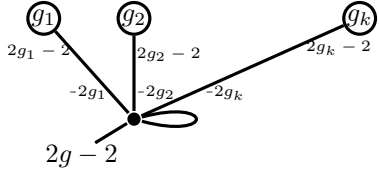
%\eas

%%%%%%%%%%%%%%%%%%%%%%%%%%%%%%%%
\subsection{Strata of quadratic differentials with odd orders of zeros only}
%%%%%%%%%%%%%%%%%%%%%%%%%%%%%%%%

The principal boundary for strata of quadratic differentials was studied in \cite{MZquad} and \cite{GoujSV}.  Note that a stratum of quadratic differentials has REL zero if and only if the order of every zero is odd.  For instance, the principal boundary configurations for the strata of quadratic differentials with simple zeros and poles were completely described in \cite[Section 4]{GoujSV}. Since this paper uses different drawing conventions, we draw the principal graphs in terms of level graphs of multi-scale {\em quadratic} differentials 
in the case when the number of singularities is large enough. (If the number of singularities is small, there exist exceptional configurations which are not included below.)
%\begin{figure}
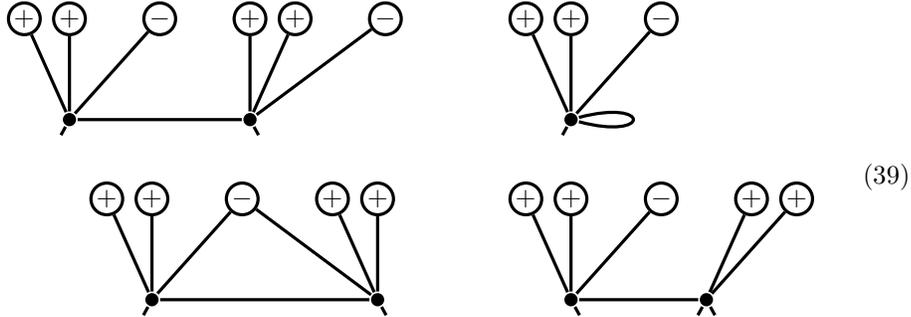
\ba   
% Pic 1
\begin{tikzpicture}[
		baseline={([yshift=-.5ex]current bounding box.center)},
		scale=2,very thick,
		bend angle=20,
		every loop/.style={very thick},
		comp/.style={circle,fill,black,,inner sep=0pt,minimum size=5pt},
		order bottom left/.style={pos=.25,left=-0.1,font=\tiny},
		order top left/.style={pos=.75,left=-0.1,font=\tiny},
		order bottom right/.style={pos=.25,right=-0.1,font=\tiny},
		order top right/.style={pos=.75,right=-0.1,font=\tiny},
		circled number/.style={circle, draw, inner sep=0pt, minimum size=12pt},
		bottom right with distance/.style={below right,text height=10pt},
		bottom mid with distance/.style={below,text height=10pt},
		bottom left with distance/.style={below left,text height=10pt},
		top right with distance/.style={above right,text height=10pt},
		top mid with distance/.style={above,text height=10pt},
		top left with distance/.style={above left,text height=10pt},
		right with distance/.style={right,text height=10pt},
		left with distance/.style={left,text height=10pt},
		ram bottom right with distance/.style={below right,font=\tiny},
		ram bottom mid with distance/.style={below,font=\tiny},
		ram bottom left with distance/.style={below left,font=\tiny},
		ram top right with distance/.style={above right,font=\tiny},
		ram top mid with distance/.style={above,font=\tiny},
		ram top left with distance/.style={above left,font=\tiny},
		ram right with distance/.style={right,font=\tiny},
		ram left with distance/.style={left,font=\tiny}]
\begin{scope}[local bounding box = a]
	\node[circled number] (T21) [xshift=3\hornodeshift] {$+$};
	\node[circled number] (T22) [xshift=4\hornodeshift] {$+$};
	\node[circled number] (T23) [xshift=6\hornodeshift] {$-$};
	\node[comp] (B2) [below=of T1, xshift=6\hornodeshift] {}
		edge (T21)
		edge (T22)
                edge (T23);
	\node [bottom right with distance] (B-2-2) at (B2.south east) {};
	\path (B2) edge [shorten >=6pt] (B-2-2.center);
	\node[circled number] (T1) [xshift=-2\hornodeshift] {$+$};
	\node[circled number] (T2) [xshift=-1\hornodeshift] {$+$};
	\node[circled number] (T3) [xshift=\hornodeshift] {$-$};
	\node[comp] (B) [below=of T1, xshift=+\hornodeshift] {}
		edge (T1)
		edge (T2)
                edge (T3)
                edge (B2);
	\node [bottom left with distance] (B-0-1) at (B.south west) {};
	\path (B) edge [shorten >=6pt] (B-0-1.center);
\end{scope}
\end{tikzpicture}
& \qquad \qquad
%Pic 2
\begin{tikzpicture}[
		baseline={([yshift=-.5ex]current bounding box.center)},
		scale=2,very thick,
		bend angle=20,
		every loop/.style={very thick},
		comp/.style={circle,fill,black,,inner sep=0pt,minimum size=5pt},
		order bottom left/.style={pos=.25,left=-0.1,font=\tiny},
		order top left/.style={pos=.75,left=-0.1,font=\tiny},
		order bottom right/.style={pos=.25,right=-0.1,font=\tiny},
		order top right/.style={pos=.75,right=-0.1,font=\tiny},
		circled number/.style={circle, draw, inner sep=0pt, minimum size=12pt},
		bottom right with distance/.style={below right,text height=10pt},
		bottom mid with distance/.style={below,text height=10pt},
		bottom left with distance/.style={below left,text height=10pt},
		top right with distance/.style={above right,text height=10pt},
		top mid with distance/.style={above,text height=10pt},
		top left with distance/.style={above left,text height=10pt},
		right with distance/.style={right,text height=10pt},
		left with distance/.style={left,text height=10pt},
		ram bottom right with distance/.style={below right,font=\tiny},
		ram bottom mid with distance/.style={below,font=\tiny},
		ram bottom left with distance/.style={below left,font=\tiny},
		ram top right with distance/.style={above right,font=\tiny},
		ram top mid with distance/.style={above,font=\tiny},
		ram top left with distance/.style={above left,font=\tiny},
		ram right with distance/.style={right,font=\tiny},
		ram left with distance/.style={left,font=\tiny}]
\begin{scope}[local bounding box = a]
	\node[circled number] (T1) [xshift=-2\hornodeshift] {$+$};
	\node[circled number] (T2) [xshift=-1\hornodeshift] {$+$};
	\node[circled number] (T3) [xshift=\hornodeshift] {$-$};
	\node[comp] (B) [below=of T1, xshift=+\hornodeshift] {}
		edge (T1)
		edge (T2)
                edge (T3)
                edge [loop right] (B);
	\node [bottom left with distance] (B-0-1) at (B.south west) {};
	\path (B) edge [shorten >=6pt] (B-0-1.center);
%	\node [bottom right with distance] (B-0-2) at (B.south east) {};
%	\path (B) edge [shorten >=6pt] (B-0-2.center);
\end{scope}
\end{tikzpicture}
\\
%Pic 3
\begin{tikzpicture}[
		baseline={([yshift=-.5ex]current bounding box.center)},
		scale=2,very thick,
		bend angle=20,
		every loop/.style={very thick},
		comp/.style={circle,fill,black,,inner sep=0pt,minimum size=5pt},
		order bottom left/.style={pos=.25,left=-0.1,font=\tiny},
		order top left/.style={pos=.75,left=-0.1,font=\tiny},
		order bottom right/.style={pos=.25,right=-0.1,font=\tiny},
		order top right/.style={pos=.75,right=-0.1,font=\tiny},
		circled number/.style={circle, draw, inner sep=0pt, minimum size=12pt},
		bottom right with distance/.style={below right,text height=10pt},
		bottom mid with distance/.style={below,text height=10pt},
		bottom left with distance/.style={below left,text height=10pt},
		top right with distance/.style={above right,text height=10pt},
		top mid with distance/.style={above,text height=10pt},
		top left with distance/.style={above left,text height=10pt},
		right with distance/.style={right,text height=10pt},
		left with distance/.style={left,text height=10pt},
		ram bottom right with distance/.style={below right,font=\tiny},
		ram bottom mid with distance/.style={below,font=\tiny},
		ram bottom left with distance/.style={below left,font=\tiny},
		ram top right with distance/.style={above right,font=\tiny},
		ram top mid with distance/.style={above,font=\tiny},
		ram top left with distance/.style={above left,font=\tiny},
		ram right with distance/.style={right,font=\tiny},
		ram left with distance/.style={left,font=\tiny}]
\begin{scope}[local bounding box = a]
	\node[circled number] (T1) [xshift=-2\hornodeshift] {$+$};
	\node[circled number] (T2) [xshift=-1\hornodeshift] {$+$};
	\node[circled number] (T3) [xshift=\hornodeshift] {$-$};
	\node[comp] (B) [below=of T1, xshift=+\hornodeshift] {}
		edge (T1)
		edge (T2)
                edge (T3);
	\node [bottom left with distance] (B-0-1) at (B.south west) {};
	\path (B) edge [shorten >=6pt] (B-0-1.center);
%	\node [bottom right with distance] (B-0-2) at (B.south east) {};
%	\path (B) edge [shorten >=6pt] (B-0-2.center);
	\node[circled number] (T21) [xshift=3\hornodeshift] {$+$};
	\node[circled number] (T22) [xshift=4\hornodeshift] {$+$};
	\node[comp] (B2) [below=of T1, xshift=6\hornodeshift] {}
		edge (T21)
		edge (T22)
                edge (T3)
                edge (B);
%	\node [bottom left with distance] (B-2-1) at (B2.south west) {};
%	\path (B2) edge [shorten >=6pt] (B-2-1.center);
	\node [bottom right with distance] (B-2-2) at (B2.south east) {};
	\path (B2) edge [shorten >=6pt] (B-2-2.center);
\end{scope}
\end{tikzpicture}
& \qquad \qquad
%Pic 4
\begin{tikzpicture}[
		baseline={([yshift=-.5ex]current bounding box.center)},
		scale=2,very thick,
		bend angle=20,
		every loop/.style={very thick},
		comp/.style={circle,fill,black,,inner sep=0pt,minimum size=5pt},
		order bottom left/.style={pos=.25,left=-0.1,font=\tiny},
		order top left/.style={pos=.75,left=-0.1,font=\tiny},
		order bottom right/.style={pos=.25,right=-0.1,font=\tiny},
		order top right/.style={pos=.75,right=-0.1,font=\tiny},
		circled number/.style={circle, draw, inner sep=0pt, minimum size=12pt},
		bottom right with distance/.style={below right,text height=10pt},
		bottom mid with distance/.style={below,text height=10pt},
		bottom left with distance/.style={below left,text height=10pt},
		top right with distance/.style={above right,text height=10pt},
		top mid with distance/.style={above,text height=10pt},
		top left with distance/.style={above left,text height=10pt},
		right with distance/.style={right,text height=10pt},
		left with distance/.style={left,text height=10pt},
		ram bottom right with distance/.style={below right,font=\tiny},
		ram bottom mid with distance/.style={below,font=\tiny},
		ram bottom left with distance/.style={below left,font=\tiny},
		ram top right with distance/.style={above right,font=\tiny},
		ram top mid with distance/.style={above,font=\tiny},
		ram top left with distance/.style={above left,font=\tiny},
		ram right with distance/.style={right,font=\tiny},
		ram left with distance/.style={left,font=\tiny}]
\begin{scope}[local bounding box = a]
	\node[circled number] (T21) [xshift=3\hornodeshift] {$+$};
	\node[circled number] (T22) [xshift=4\hornodeshift] {$+$};
	\node[comp] (B2) [below=of T1, xshift=4\hornodeshift] {}
		edge (T21)
		edge (T22);
	\node [bottom right with distance] (B-2-2) at (B2.south east) {};
	\path (B2) edge [shorten >=6pt] (B-2-2.center);
	\node [bottom left with distance] (B-2-1) at (B2.south west) {};
	\path (B2) edge [shorten >=6pt] (B-2-1.center);
        \node[circled number] (T1) [xshift=-2\hornodeshift] {$+$};
	\node[circled number] (T2) [xshift=-1\hornodeshift] {$+$};
	\node[circled number] (T3) [xshift=\hornodeshift] {$-$};
	\node[comp] (B) [below=of T1, xshift=+\hornodeshift] {}
		edge (T1)
		edge (T2)
                edge (T3)
                edge (B2);
	\node [bottom right with distance] (B-0-2) at (B.south east) {};
	\path (B) edge [shorten >=6pt] (B-0-2.center);
\end{scope}
\end{tikzpicture}
%\caption{Typical (quadratic) principal graphs of strata of quadratic differentials with odd singularities}
%	\label{fig:quadodd}
%\end{figure}
\ea

In these schematic pictures, a vertex with a plus-sign indicates that the corresponding quadratic differential is a square of an Abelian differential, while a minus-sign indicates that the differential is strictly quadratic, i.e., not a global square of an Abelian differential. We do not indicate the zeros nor poles adjacent to the vertices, nor their genera of any of the vertices on the top level. The number of Abelian vertices can be arbitrary, zeros permitted, but the number of strictly quadratic vertices is either one or two, as indicated. Finally, using the (admissible) canonical double cover of these quadratic differentials, one can also draw the principal graphs of the corresponding multi-scale {\em Abelian} differentials when lifting the strata of quadratic differentials as affine invariant submanifolds into the respective strata of Abelian differentials. 

%%%%%%%%%%%%%%%%%%%%%%%%%%%%%%%%
\subsection{The flex and gothic loci}
%%%%%%%%%%%%%%%%%%%%%%%%%%%%%%%%
We start with a brief recap on the gothic locus $\Omega \mathcal{G}$, an exceptional affine invariant submanifold found in \cite{MMWgothic}, together with its companions, the flex locus~$\cF$ and the corresponding locus of quadratic differentials~$\cQ\cF$ over $\cF$, where $\Omega \mathcal{G}\to \cQ\cF$ can be identified through the canonical double covering construction (up to labeling the zeros and poles).  
 
An example of a flat surface parameterized in the gothic locus is given in Figure~\ref{fig:gothicflat}.
\begin{figure}[htb]
	\centering
		\begin{tikzpicture}[auto]
			\draw ($(0,0) + (60:-1)$)
				-- node {$v_4$} ++(0:1) node[nodecirc,fill=white] {}
				-- node {$v_3$} ++(300:1) node[nodecirc,fill=gray] {}
				-- node {$v_2$} ++(240:1) node[nodecirc,fill] {}
				-- node {$v_1$} ++(180:1) node[nodecirc,fill=white] {}
				-- node {$v_6$} ++(120:1) node[nodecirc,fill=gray] {}
				-- node {$v_5$} ++(60:1) node[nodecirc,fill] {};
			\draw (3,0)
				-- node {$w_1$} ++(0:1) node[nodecirc,fill=white] {}
				-- node {$v_1$} ++(0:1) node[nodecirc,fill] {}
				-- node {$w_6$} ++(300:1) node[nodecirc,fill=gray] {}
				-- node {$v_6$} ++(300:1) node[nodecirc,fill=white] {}
				-- node {$w_5$} ++(240:1) node[nodecirc,fill] {}
				-- node {$v_5$} ++(240:1) node[nodecirc,fill=gray] {}
				-- node {$w_4$} ++(180:1) node[nodecirc,fill=white] {}
				-- node {$v_4$} ++(180:1) node[nodecirc,fill] {}
				-- node {$w_3$} ++(120:1) node[nodecirc,fill=gray] {}
				-- node {$v_3$} ++(120:1) node[nodecirc,fill=white] {}
				-- node {$w_2$} ++(60:1) node[nodecirc,fill] {}
				-- node {$v_2$} ++(60:1) node[nodecirc,fill=gray] {};
			\draw ($(8,0) + (60:-1)$)
				-- node {$w_4$} ++(0:1) node[nodecirc,fill=gray] {}
				-- node {$w_3$} ++(300:1) node[nodecirc,fill] {}
				-- node {$w_2$} ++(240:1) node[nodecirc,fill=white] {}
				-- node {$w_1$} ++(180:1) node[nodecirc,fill=gray] {}
				-- node {$w_6$} ++(120:1) node[nodecirc,fill] {}
				-- node {$w_5$} ++(60:1) node[nodecirc,fill=white] {};
		\end{tikzpicture}
	\caption{A flat surface parameterized in the locus of cyclic forms inside the gothic locus (from \cite[Figure~4]{MMWgothic})}
	\label{fig:gothicflat}
\end{figure}
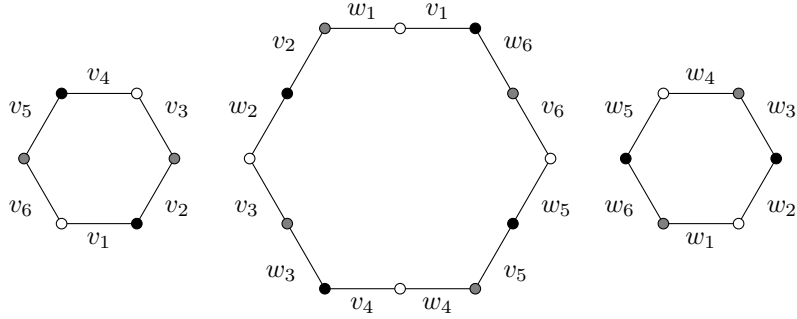

For the intersection calculations on the flex locus, we will use the notation and results in \cite{chengothic}. Let $\ol{\cF}\subset \ol{\cM}_{1,3}$ be the closure of the flex locus parameterizing $(E, p_1, p_2, p_3)$, where 
$E$ admits a plane cubic model with a projection $\pi$ from a point $s \in \bP^2$ such that $p_1$, $p_2$, and $p_3$ are three collinear cocritical points of $\pi$. Let $\cQ \cF\subset \cQ_{1,3}(-1^3,1^3)$ be the locus of quadratic differentials $(E, q)$ over $\cF$ such that $\div (q) = z_1 + z_2 + z_3  - p_1 - p_2 - p_3$, where $z_1 + z_2 + z_3$ is a fiber of $\pi$. We have $\dim_{\bC} \cQ \cF = 4$. Note that we follow the setting of \cite{chengothic} for $\cQ \cF$, where $p_1$, $p_2$, and $p_3$ are ordered but $z_1$, $z_2$, and $z_3$ are unordered.  Lifting via the canonical double cover identifies $\cQ \cF$ with the gothic locus $\Omega \mathcal{G} \subset \Omega \cM_{4,6} (2^3, 0^3)$.
In order to keep it consistent with the setting in this paper, we label all the preimages of $p_i$ and $z_j$ in $\Omega \mathcal{G}$. Therefore, $\PP\Omega \mathcal{G}\to  \PP \cQ \cF$ is a finite map of degree six, where each fiber consists of the choices of labeling the preimages of $z_1$, $z_2$, and $z_3$. 
We also denote by $\PP\Omega \ol{\mathcal{G}}$ and $\PP\cQ \ol{\cF}$ the corresponding closures in the multi-scale compactification.   
\par
Combining the intersection numbers obtained in \cite{chengothic} with our main Theorem~\ref{thm:main}, we restate and prove the part of Corollary~\ref{cor:introgothic} regarding the value of $c_{\rm area}(\Omega \mathcal{G})$  as follows. 
\par
\begin{cor} \label{cor:SVflex}
The area Siegel--Veech constant of the gothic locus $\Omega \mathcal{G}$ is 
$$c_{\rm area}(\Omega \mathcal{G}) \= \frac{1}{\pi^2}\cdot \frac{49}{13}\,.$$
\end{cor}
\par
\begin{proof}
We can verify the claim in two ways. First, combining \cite[Theorem~1.2]{chengothic}, the quadratic version of Theorem~\ref{thm:main} as explained in \cite{CMSprincipal}, and \cite[Equation~(2.4)]{ekz}, we find
$$ L^{+}(\cQ \cF) \= \frac{6}{13}\,, \quad L^{-}(\cQ \cF) \= \frac{19}{13}\,,$$
where $L^{\pm}$ stands for the respective sum of Lyapunov exponents. It follows that 
 $$  L (\Omega \mathcal{G}) \= L^+(\cQ\cF) + L^-(\cQ\cF) \=  \frac{25}{13}\,.$$
Now, using \cite[Equation~(2.1)]{ekz}, we thus obtain 
$$ c_{\area} (\Omega \mathcal{G}) \=  \frac{1}{\pi^2}\cdot \frac{49}{13}\,. $$
\par 
Alternatively, we can use the intersection formula in Theorem~\ref{thm:main} directly. Let~$\eta$ be the divisor class of the $\calO(-1)$ bundle on $\bP\cQ\ol{\cF}$. Then \cite[Theorem~1.2]{chengothic} implies that 
\be \label{eq:chengothic}
\int_{\bP\cQ\ol{\cF}} \eta^3 \= -\frac{13}{6}\,, \qquad  \int_{\bP\cQ\ol{\cF}} \eta^2 \lambda_1 \= -\frac{1}{2}\,.
\ee
By the well-known relation of divisor classes $12 \lambda = \kappa + \delta$ on the moduli space of curves 
and 
$ \kappa = - \eta$ (modulo non-horizontal boundary divisor classes) on $\bP\cQ\ol{\cF}$ (see \cite[Proposition~2.1]{chentauto}), we obtain  
\be
\int_{\bP\cQ\ol{\cF}} \eta^2 \delta \= -\frac{49}{6}\,.
\ee
Next, we lift the above intersection calculations to the gothic locus $\PP\Omega \ol{\mathcal{G}}$ through the canonical double cover. In this process, note that $\eta$ pulls back 
to $-2\xi$, we label (the preimages of) $z_1$, $z_2$, and $z_3$ in $\PP\Omega \ol{\mathcal{G}}$ but not in $\bP\cQ\ol{\cF}$, differentials in $\PP\Omega \ol{\mathcal{G}}$ has an involution automorphism of order two due to the double covering construction, and every horizontal node in the target of 
 an admissible double cover has two preimage nodes in the domain. Taking all of these into account, we thus obtain 
\be \label{eq:liftgothic}
\int_{\PP\Omega \ol{\mathcal{G}}} \xi^3 \= \frac{13}{16}\,, \qquad  \int_{\PP\Omega \ol{\mathcal{G}}} \xi^2 \delta \= -\frac{49}{4}\,.
\ee
Finally, we conclude that 
$$c_{\rm area}(\Omega \mathcal{G})  \= - \frac{1}{4\pi^2} \cdot \frac{ \int_{\PP\Omega \ol{\mathcal{G}}} \xi^2 \delta}{\int_{\PP\Omega \ol{\mathcal{G}}} \xi^3} \= \frac{1}{\pi^2}\cdot \frac{49}{13}\, $$
by using Theorem~\ref{thm:main}.
\end{proof}
\par 
In the remaining part of this section, we will describe all the principal graphs for the gothic locus, which consist of four types and are labeled by I, II, III, and IV below.
\input{gothic_principalgraphs}
\par 
This provides an alternative way to evaluate the area Siegel--Veech constant of $\Omega \mathcal{G}$ by decomposing it into the respective contribution of each principal boundary component as in~\eqref{eq:one-loop}, thus proving the remaining part of Corollary~\ref{cor:introgothic}. 
\par 
\begin{prop}\label{prop:SVgothicCases}
For each type of the principal boundary components of $\Omega \mathcal{G}$, its contribution to $c_{\area} (\Omega \mathcal{G})$ defined in \eqref{eq:one-loop} is: 
\ba \label{eq:SVgothicCases}
c_{\rm area}^{\rm int}(\Omega \mathcal{G}, \Gamma_{\rm I}) = \frac{1}{\pi^2} \cdot \frac{15}{13}\,, \quad
c_{\rm area}^{\rm int}(\Omega \mathcal{G}, \Gamma_{\rm II})  = \frac{1}{\pi^2}  \cdot \frac{1}{13}\,, \\ 
c_{\rm area}^{\rm int} (\Omega \mathcal{G}, \Gamma_{\rm III}) = \frac{1}{\pi^2}  \cdot \frac{27}{13}\,, \quad
c_{\rm area}^{\rm int} (\Omega \mathcal{G}, \Gamma_{\rm IV}) = \frac{1}{\pi^2}  \cdot \frac{6}{13}\,.
\ea
In particular, they sum up to the value of $c_{\area}(\Omega \mathcal{G})$ as in Corollary~\ref{cor:SVflex}. 
\end{prop}
\par
Let ${\rm X}$ denote one of the above four types. The strategy of evaluating $c_{\rm area}^{\rm int}(\Omega \mathcal{G}, \Gamma_{\rm X})$ is the following. First,  \eqref{eq:one-loop} reduces to 
\be \label{eq:gothic-int-cases}
c_{\rm area}^{\rm int}(\Omega \mathcal{G}, \Gamma_{\rm X}) = -\frac{1}{4\pi^2}\cdot c_{\Gamma_{\rm X}}^{\area} \cdot \cK_{\cY_{\Gamma_{\rm X}}}\cdot  \frac{\int_{\PP\cY^{\top}_{\Gamma_{\rm X}}}\xi }{\int_{\PP \Omega \mathcal{\ol{G}}} \xi^3}\,.
\ee
Additionally, we know $\int_{\PP \Omega \mathcal{\ol{G}}} \xi^3 = 13/16$ by~\eqref{eq:liftgothic}. Moreover, for all the principal graphs of $\Omega \mathcal{G}$, 
the number of reachable prong-matching equivalence classes (as part of the data to determine $\cK_{\cY_{\Gamma_{\rm X}}}$) is always $1$. Therefore, it remains to determine the factor $c_{\Gamma_{\rm X}}^{\area}$,  the top-level degree $\int_{\PP\cY^{\top}_{\Gamma_{\rm X}}}\xi$, and the graph automorphism subgroup, all of which can be solved by using the explicit geometry of each principal graph.  
\par
We now discuss the cases of possible principal graphs and nickname them according to the shapes of the underlying two-level graphs without horizontal degeneration. There are two ways to prove that the list in~\eqref{eq:gothicprincgraphs} is complete. First, V.~Delecroix has provided us with a list of boundary intersections obtained by searching the full boundary of the gothic locus using the sage-program \texttt{flatsurf}, which consists of those types we have found using admissible covers and degenerations studied by J.~Schwab in \cite{Schwab2024}. Alternatively, the sum of the area Siegel--Veech contributions in~\eqref{eq:SVgothicCases} is already equal to the value of $c_{\rm area}(\Omega \mathcal{G})$ in Corollary~\ref{cor:SVflex}, which implies that there cannot be any other principal boundary contributions. 

%%%%%%%%%%%%%%%%%%%%%%%%%%%%%%%%%%%%%%
\subsubsection{Type I: The banana graph}
%%%%%%%%%%%%%%%%%%%%%%%%%%%%%%%%%%%%%%

In order to understand the top level, we use a flat geometric degeneration. We start from Figure~\ref{fig:gothicflat} and recall from \cite{MMWgothic} that the gothic locus is locally cut out by the equations
\ba \label{eq:gothicequations}
&v_i = -v_{i+3}\,, \quad w_i = - w_{i+3}\,, \qquad i=1,2,3 \\
&v_1+v_3+v_5 = w_1 +w_3 + w_5 =0\,.
\ea
We deform the middle $12$-gon until its shape approaches the swiss cross on the left of Figure~\ref{fig:G5top}. In particular, as~$w_1$ becomes parallel to~$w_6$ (and thus~$w_4$ parallel to~$w_3$) and at the same time~$v_5$ and~$v_6$ become parallel (and thus~$v_2$ and~$v_3$), the small hexagons disappear and this swiss cross (or the regluing on the right) is actually the full top level and in the swiss cross these pairs of sides should be identified. The small crosses inside the swiss cross are the regular fixed points of the double cover involution. The gray point is the unique double zero visible on the top level. The black and white squares are the newly formed (regular) fixed points, which thus become the zeros at the top of the edges that connect to the lower level. The banana graph in~$\Gamma_{\rm I}$ (after smoothing the horizontal edge) is the unique level graph that is consistent with the aforementioned data.
\begin{figure}[ht]
	\centering
		\begin{tikzpicture}[auto]
			\draw (0,0)
				-- node {$v_2$} ++(1,0) node[nodecirc,fill=gray] {}
				-- node {$w_1$} ++(60:1) node[noderect,fill=white] {}
				-- node {$v_1$} ++(2,0) node[noderect,fill=white] {}
				-- node[swap] {$w_6$} ++(60:-1) node[nodecirc,fill=gray] {}
				-- node {$v_6$} ++(1,0) node[noderect,fill] {}
				-- node {$w_5$} ++(60:-2) node[noderect,fill] {}
				-- node {$v_5$} ++(-1,0) node[nodecirc,fill=gray] {}
				-- node {$w_4$} ++(60:-1) node[noderect,fill=white] {}
				-- node {$v_4$} ++(-2,0) node[noderect,fill=white] {}
				-- node[swap] {$w_3$} ++(60:1) node[nodecirc,fill=gray] {}
				-- node {$v_3$} ++(-1,0) node[noderect,fill] {}
				-- node {$w_2$} ++(60:2) node[noderect,fill] {};
			\node at (60:-1) {$\times$};
			\node at ($(60:-1) + (4,0)$) {$\times$};
			\node at ($(60:-1) + (2,0)$) {$\times$};
			\node at ($(60:1) + (2,0)$) {$\times$};
			\node at ($(60:-3) + (2,0)$) {$\times$};

			\draw[dashed] (8,-1) -- (10,-1);
			\draw[dashed] (8,-1) -- ++(60:-2);
			\draw (6,-1)
				-- node {$\alpha_4$} ++(2,0) node[nodecirc,fill=gray] {}
				-- node {$\alpha_1$} ++(60:2) node[nodecirc,fill=gray] {}
				-- node {$\alpha_3$} ++(2,0) node[nodecirc,fill=gray] {}
				-- node {$\alpha_1$} ++(60:-2) node[nodecirc,fill=gray] {}
				-- node {$\alpha_2$} ++(60:-2) node[nodecirc,fill=gray] {}
				-- node {$\alpha_3$} ++(-2,0) node[nodecirc,fill=gray] {}
				-- node {$\alpha_4$} ++(-2,0) node[nodecirc,fill=gray] {}
				-- node {$\alpha_2$} ++(60:2) node[nodecirc,fill=gray] {};
			\node at ($(60:-1) + (7,-1)$) {$\times$};
			\node at ($(60:-1) + (9,-1)$) {$\times$};
			\node at ($(60:1) + (9,-1)$) {$\times$};
			\node[nodecirc,fill] at (7,-1) {};
			\node[nodecirc,fill] at ($(7,-1) + (60:-2)$) {};
			\node[nodecirc,fill=white] at ($(60:1) + (8,-1)$) {};
			\node[nodecirc,fill=white] at ($(60:1) + (10,-1)$) {};
		\end{tikzpicture}
	\caption{The top-level component of Type I and a regluing (sides with the same symbol~$\alpha_j$ being identified)}
	\label{fig:G5top}
\end{figure}
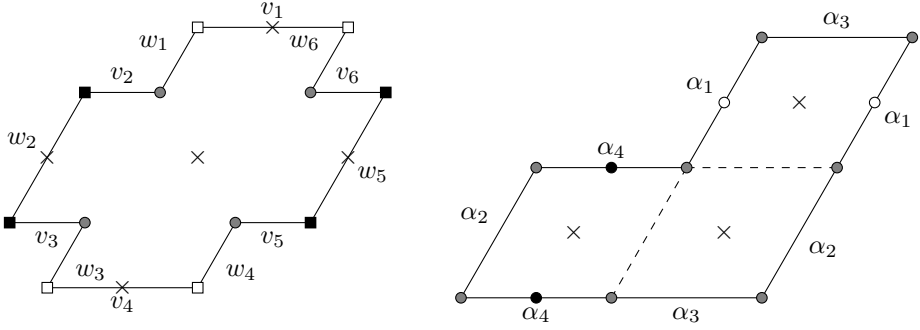

We can now compute the top-level volume contribution. Disregarding the marked points, the top level consists of the square-tiled surface~$L_3$ on the right of Figure~\ref{fig:G5top}. Its Veech group is the theta group $\Gamma_\theta = \langle \sm1201, \sm 01{-1}0\rangle$ which has index three in $\SL_2(\bZ)$. Consider the action of the theta group on the five Weierstrass points, where the differential form is regular (i.e., has a ``zero of order zero''). It is transitive and the image is a dihedral group $D_5$ of order~$10$. Since in our setting all the special points are marked, the top level is a degree-$10$ cover of the Teichm\"uller curve generated by the~$L_3$ surface.
Since $\int_{\PP \Omega\BM_{1,1}(0)} \xi = -1/24$, we conclude that 
\be
\int_{\PP\cY^\top_{\Gamma_{\rm I}}}\xi \= 10\cdot 3 \cdot \frac{-1}{24}  \= \frac{-5}{4}\,. 
\ee
\par
In order to understand the sizes of the cylinders that are formed at the lower level, we look more closely at the degenerating hexagons. This lower level is given (in the neighborhood near the limit) by cutting out along the pinched curve, as shown in Figure~\ref{fig:lowerlevel} on the left. After redrawing, splitting each of the shaded regions in the middle, and taking these to the limit, the lower level (without yet having acquired a horizontal node) is given on the right of  Figure~\ref{fig:lowerlevel}. The lower level acquires a cylinder (in fact a pair of cylinders, exchanged by the involution, hence $\calM$-parallel cylinders),  once the lower level is deformed so that $v \wedge y$ changes sign. 
\begin{figure}[htb]
	\centering

\definecolor{zzttqq}{rgb}{0.6,0.2,0}
\begin{minipage}{0.49\textwidth}
  %\begin{scope}
  \begin{tikzpicture}[line cap=round,line join=round,>=triangle 45,x=1.5cm,y=1.2cm]
%    xtick={-2,-1.5,...,3.5},
%ytick={-2,-1.5,...,2},]
%\clip(-2.201945048962654,-2.2498441070276316) rectangle (3.695473479596204,2.35863299117837);
\fill[line width=0.5pt,color=gray,fill=black,fill opacity=0.10000000149011612] (0.8,0.4) -- (1,0.5) -- (1.2,0.5) -- (1.2,-0.5) -- (1,-0.5) -- (0.8,-0.4) -- cycle;
\fill[line width=0.5pt,color=gray,fill=black,fill opacity=0.10000000149011612] (1.8,0.5) -- (2,0.5) -- (2.2,0.4) -- (2.2,-0.4) -- (2,-0.5) -- (1.8,-0.5) -- cycle;
\fill[line width=0.5pt,color=gray,fill=black,fill opacity=0.10000000149011612] (-1.2532478254956305,-0.14646356398538757) -- (-1.2,0) -- (-1.1021905530218987,0.10806308193218539) -- (-0.7345658598645561,-0.21111147688593543) -- (-0.8515254104639517,-0.33374777669745725) -- (-0.9895959720021429,-0.3901031079375353) -- cycle;
\fill[line width=0.5pt,color=gray,fill=black,fill opacity=0.10000000149011612] (-0.8764892533572628,0.3574252733299245) -- (-0.7927498480264397,0.44994331221877804) -- (-0.6304013255037689,0.502280929984898) -- (-0.3585742570173689,0.270001101556312) -- (-0.40084180536476105,0.13881020964767418) -- (-0.5009453396209492,0.03384804985549822) -- cycle;

\draw [line width=0.5pt,color=gray] (0.8,0.4)-- (1,0.5);
\draw [line width=0.5pt,color=black] (1,0.5)-- (1.2,0.5);
\draw [line width=0.5pt,color=black] (1.2,0.5)-- (1.2,-0.5);
\draw [line width=0.5pt,color=black] (1.2,-0.5)-- (1,-0.5);
\draw [line width=0.5pt,color=black] (1,-0.5)-- (0.8,-0.4);
\draw [line width=0.5pt,color=black] (0.8,-0.4)-- (0.8,0.4);
\draw [line width=0.5pt,color=black] (1.8,0.5)-- (2,0.5);
\draw [line width=0.5pt,color=black] (2,0.5)-- (2.2,0.4);
\draw [line width=0.5pt,color=black] (2.2,0.4)-- (2.2,-0.4);
\draw [line width=0.5pt,color=black] (2.2,-0.4)-- (2,-0.5);
\draw [line width=0.5pt,color=black] (2,-0.5)-- (1.8,-0.5);
\draw [line width=0.5pt,color=black] (1.8,-0.5)-- (1.8,0.5);
\draw [line width=0.5pt] (0,0) node[nodecirc,fill=gray] {} -- (1,0.5) node[nodecirc,fill=black] {};
\draw [line width=0.5pt] (1,0.5)-- (2,0.5) node[nodecirc,fill=white] {};
\draw [line width=0.5pt] (2,0.5)-- (3,0) node[nodecirc,fill=gray] {};
\draw [line width=0.5pt] (3,0)-- (2,-0.5) node[nodecirc,fill=black] {};
\draw [line width=0.5pt] (2,-0.5)-- (1,-0.5) node[nodecirc,fill=white] {};
\draw [line width=0.5pt] (1,-0.5)-- (0,0) node[nodecirc,fill=gray] {};
\draw [line width=0.5pt] (1.2,0.5)-- (1.2,-0.5);
\draw [line width=0.5pt] (0.8,-0.4)-- (0.8,0.4); 
\draw [line width=0.5pt] (1.8,0.5)-- (1.8,-0.5);
\draw [line width=0.5pt] (2.2,0.4)-- (2.2,-0.4);
\draw [line width=0.5pt] (-1.403199485724736,-0.5589208686406344)-- (-1.2,0) node[nodecirc,fill=black] {};
\draw [line width=0.5pt] (-1.2,0)-- (-0.7927498480264397,0.44994331221877804) node[nodecirc,fill=white] {};
\draw [line width=0.5pt] (-0.7927498480264397,0.44994331221877804)-- (-0.24355806221737297,0.6269907489267652) node[nodecirc,fill=gray] {};
\draw [line width=0.5pt] (-0.24355806221737297,0.6269907489267652)-- (-0.40084180536476105,0.13881020964767418) node[nodecirc,fill=black] {};
\draw [line width=0.5pt] (-0.40084180536476105,0.13881020964767418)-- (-0.8515254104639517,-0.33374777669745725) node[nodecirc,fill=white] {};
\draw [line width=0.5pt] (-0.8515254104639517,-0.33374777669745725)-- (-1.403199485724736,-0.5589208686406344) node[nodecirc,fill=gray] {};
\draw [line width=0.5pt] (-1.2532478254956305,-0.14646356398538757)-- (-0.9895959720021429,-0.3901031079375353);
\draw [line width=0.5pt] (-1.1021905530218987,0.10806308193218539)-- (-0.7345658598645561,-0.21111147688593543);
\draw [line width=0.5pt] (-0.8764892533572628,0.3574252733299245)-- (-0.5009453396209492,0.03384804985549822);
\draw [line width=0.5pt] (-0.6304013255037689,0.502280929984898)-- (-0.3585742570173689,0.270001101556312);
\draw [line width=0.5pt,color=black] (-1.2532478254956305,-0.14646356398538757)-- (-1.2,0);
\draw [line width=0.5pt,color=black] (-1.2,0)-- (-1.1021905530218987,0.10806308193218539);
\draw [line width=0.5pt,color=black] (-1.1021905530218987,0.10806308193218539)-- (-0.7345658598645561,-0.21111147688593543);
\draw [line width=0.5pt,color=black] (-0.7345658598645561,-0.21111147688593543)-- (-0.8515254104639517,-0.33374777669745725);
\draw [line width=0.5pt,color=black] (-0.8515254104639517,-0.33374777669745725)-- (-0.9895959720021429,-0.3901031079375353);
\draw [line width=0.5pt,color=black] (-0.9895959720021429,-0.3901031079375353)-- (-1.2532478254956305,-0.14646356398538757);
\draw [line width=0.5pt,color=black] (-0.8764892533572628,0.3574252733299245)-- (-0.7927498480264397,0.44994331221877804);
\draw [line width=0.5pt,color=black] (-0.7927498480264397,0.44994331221877804)-- (-0.6304013255037689,0.502280929984898);
\draw [line width=0.5pt,color=black] (-0.6304013255037689,0.502280929984898)-- (-0.3585742570173689,0.270001101556312);
\draw [line width=0.5pt,color=black] (-0.3585742570173689,0.270001101556312)-- (-0.40084180536476105,0.13881020964767418);
\draw [line width=0.5pt,color=black] (-0.40084180536476105,0.13881020964767418)-- (-0.5009453396209492,0.03384804985549822);
\draw [line width=0.5pt,color=black] (-0.5009453396209492,0.03384804985549822)-- (-0.8764892533572628,0.3574252733299245);
\begin{scriptsize}
\draw[color=black] (1.5,0.6) node {$v_4$};
\draw[color=black] (1.5,-0.6) node {$v_1$};
\draw[color=black] (2.55,-0.4) node {$v_2$};
\draw[color=black] (2.55,+0.4) node {$v_3$};
\draw[color=black] (0.45,-0.4) node {$v_6$};
\draw[color=black] (0.45,+0.4) node {$v_5$};
\draw[color=black] (-1.1,0.35) node {$w_5$};
\draw[color=black] (-1.1,-0.6) node {$w_1$};
\draw[color=black] (-0.5,0.65) node {$w_4$};
\draw[color=black] (-0.5,-0.2) node {$w_2$};
\draw[color=black] (-0.15,+0.3) node {$w_3$};
\draw[color=black] (-1.5,-0.3) node {$w_6$};
\end{scriptsize}
\end{tikzpicture}
\end{minipage}
\begin{minipage}{0.49\textwidth}
  \begin{tikzpicture}[line cap=round,line join=round,>=triangle 45,x=1.5cm,y=1.2cm]
    \draw[color=white] (0,0) node {$5$};
    \begin{scope}[shift={(1.2,0)}]
\clip(-0.627810755714013,-1.205982284118032) rectangle (3.25357614757526,1.206615244032934);
\fill[line width=0pt,color=white,fill=black,fill opacity=0.10000000149011612] (0.5,0.5) -- (0.5,-0.5) -- (0.5,-1) -- (0.1,-1) -- (0.1,0.7) -- (1.2,0.7) -- cycle;
\fill[line width=0pt,color=white,fill=black,fill opacity=0.10000000149011612] (1.2,0.7) -- (1.2,-1) -- (0.9,-0.8) -- (0.8920634920634994,0.2) -- (0.5,0.5) -- cycle;
\fill[line width=0pt,color=white,fill=black,fill opacity=0.10000000149011612] (1.2,-1) -- (0.9,-0.8) -- (0.5,-0.5) -- (0.5,-1) -- cycle;
\fill[line width=0pt,color=white,fill=black,fill opacity=0.10000000149011612] (2,0.5) -- (2,-0.5) -- (2,-1) -- (1.6,-1) -- (1.6,0.7) -- (2.7,0.7) -- cycle;
\fill[line width=0pt,color=white,fill=black,fill opacity=0.10000000149011612] (2,0.5) -- (2.7,0.7) -- (2.7,-1) -- (2.4,-0.8) -- (2.4,0.2) -- cycle;
\fill[line width=0pt,color=white,fill=black,fill opacity=0.10000000149011612] (2,-0.5) -- (2.4,-0.8) -- (2.7,-1) -- (2,-1) -- cycle;
\draw [line width=0.5pt] (0.5,0.5) node[nodecirc,fill=white] {} -- (0.5,-0.5) node[nodecirc,fill=black] {};
\draw [line width=0.5pt] (0.5,-0.5)-- (0.9,-0.8);
\draw [line width=0.5pt] (0.9,0.2) node[nodecirc,fill=black] {} -- (0.9,-0.8) node[nodecirc,fill=white] {};
\draw [line width=0.5pt] (0.3,0.5)-- (0.5,0.5);
\draw [line width=0.5pt] (0.5,0.5)-- (0.7,0.65);
\draw [line width=0.5pt] (0.1,0.7)-- (0.1,-1);
\draw [line width=0.5pt] (0.1,-1)-- (1.2,-1);
\draw [line width=0.5pt] (1.2,-1)-- (1.2,0.7) ;
\draw [line width=0.5pt] (1.2,0.7)-- (0.1,0.7);
\draw [line width=0.5pt,color=black] (0.5,0.5)-- (0.5,-0.5);
%\draw [line width=0.5pt,color=black] (0.5,-0.5)-- (0.5,-1);
\draw [line width=0.5pt,color=black] (0.5,-1)-- (0.1,-1);
\draw [line width=0.5pt,color=black] (0.1,-1)-- (0.1,0.7);
\draw [line width=0.5pt,color=black] (0.1,0.7)-- (1.2,0.7);
%\draw [line width=0.5pt,color=black] (1.2,0.7)-- (0.5,0.5);
\draw [line width=0.5pt,color=black] (1.2,0.7)-- (1.2,-1);
%\draw [line width=0.5pt,color=black] (1.2,-1)-- (0.9,-0.8);
\draw [line width=0.5pt,color=black] (0.9,-0.8)-- (0.9,0.2);
\draw [line width=0.5pt,color=black] (0.9,0.2)-- (0.5,0.5);
\draw [line width=0.5pt,color=black] (0.9,-0.8)-- (0.5,-0.5);
\draw [line width=0.5pt,color=black] (0.5,-1)-- (1.2,-1);
\draw [line width=0.5pt] (1.6,0.7)-- (1.6,-1);
\draw [line width=0.5pt] (2.7,-1)-- (2.7,0.7);
\draw [line width=0.5pt] (2.7,0.7)-- (1.6,0.7);
\draw [line width=0.5pt] (1.6,-1)-- (2.7,-1);
\draw [line width=0.5pt] (2,0.5) node[nodecirc,fill=black] {} -- (2,-0.5) node[nodecirc,fill=white] {};
\draw [line width=0.5pt] (2,0.5)-- (2.4,0.2) ;
\draw [line width=0.5pt] (2,-0.5)-- (2.4,-0.8);
\draw [line width=0.5pt] (2.4,0.2)-- (2.4,-0.8);
\draw [line width=0.5pt] (1.8,0.5)-- (2,0.5);
\draw [line width=0.5pt] (2,0.5)-- (2.2,0.65);
\draw [line width=0.5pt] (2.2,-0.95)-- (2.4,-0.8);
\draw [line width=0.5pt] (2.4,-0.8)-- (2.6,-0.8);
\draw [line width=0.5pt,color=black] (2,-1)-- (1.6,-1);
\draw [line width=0.5pt,color=black] (1.6,-1)-- (1.6,0.7);
\draw [line width=0.5pt,color=black] (1.6,0.7)-- (2.7,0.7);
\draw [line width=0.5pt,color=black] (2.7,0.7)-- (2.7,-1);
%\draw [line width=0.5pt,color=black] (2.7,-1)-- (2.4,-0.8);
\draw [line width=0.5pt,color=black] (2.4,-0.8) node[nodecirc,fill=black] {}-- (2.4,0.2) node[nodecirc,fill=white] {};
\draw [line width=0.5pt,color=black] (2.4,0.2)-- (2,0.5);
\draw [line width=0.5pt,color=black] (2,-0.5)-- (2.4,-0.8);
%\draw [line width=0.5pt,color=black] (2.4,-0.8)-- (2.7,-1);
\draw [line width=0.5pt,color=black] (2.7,-1)-- (2,-1);
%\draw [line width=0.5pt,color=black] (2,-1)-- (2,-0.5);
\draw [line width=0.5pt] (0.7,-0.95)-- (0.9,-0.8);
\draw [line width=0.5pt] (0.9,-0.8)-- (1.1,-0.8);
\draw [line width=0.5pt] (0.9,0.2)-- (1.1,0.2);
\draw [line width=0.5pt] (0.5,-0.5)-- (0.3,-0.5);
\draw [line width=0.5pt] (2,-0.5)-- (1.8,-0.5);
\draw [line width=0.5pt] (2.4,0.2)-- (2.6,0.2);
\begin{scriptsize}
%\draw [fill=black] (0.5,0.5) circle (2.5pt);
\draw[color=black] (0.75,0.55) node {$w_4$};
\draw[color=black] (0.35,0.57) node {$v_4$};
\draw[color=black] (0.4,0) node {$x$};
\draw[color=black] (0.35,-0.4) node {$v_5$};
\draw[color=black] (1.0,-0.2) node {$y$};
\draw[color=black] (1.05,0.3) node {$v_4$};
\draw[color=black] (1.05,-0.7) node {$v_1$};
\draw[color=black] (0.7,0.2) node {$u$};
\draw[color=black] (0.7,-0.5) node {$z$};
\draw[color=black] (0.7,-0.85) node {$w_1$};
\draw[color=black] (2.25,0.5) node {$w_5$};
\draw[color=black] (1.85,0.57) node {$v_5$};
\draw[color=black] (1.9,0) node {$y$};
\draw[color=black] (2.5,-0.2) node {$x$};
\draw[color=black] (2.2,0.2) node {$z$};
\draw[color=black] (2.2, -0.5) node {$u$};
\draw[color=black] (2.55,+0.3) node {$v_3$};
\draw[color=black] (2.55,-0.65) node {$v_2$};
\draw[color=black] (2.2,-0.8) node {$w_2$};
\draw[color=black] (1.85,-0.4) node {$v_6$};
\end{scriptsize}
\end{scope}
\end{tikzpicture}
\end{minipage}
	\caption{Left: A neighborhood (near the limit) of the lower level. (Small arcs around the points in the central $12$-gon are not drawn.) Right: The limiting lower level (in gray) after regluing. (The two planes should be extended to infinity to form poles.)}
	\label{fig:lowerlevel}
\end{figure}
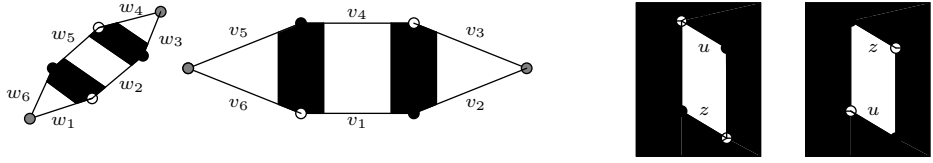

\par
We now choose a basis according to the requirements of Lemma~\ref{le:adaptedbasis}. We take $\alpha_1 = z+y$ to be the core curve of the cylinder, and take $\beta_1 = 2z$ to be a symplectically dual curve, normalized with some hindsight. We compute from  Figure~\ref{fig:gothicflat} and the equations in~\eqref{eq:gothicequations} that the total area of a surface in the gothic locus is
\be
\mathrm{Area} = 6 v_1\wedge v_2 + 6 w_2 \wedge w_3 + 4v_1\wedge w_2  - 2 v_1\wedge w_3  - 2v_2 \wedge w_2 + 4v_2 \wedge w_3\,,
\ee
where we use the same letter to denote both a path and its period. We now take
\be
\alpha_2 \= -3v_1 + 3v_2 + 3w_3 \quad \text{and} \quad \beta_2 \= -v_1-v_2-2w_2-3w_3\,. 
\ee
Since by construction $z = 2w_1-w_2$ and $y = 2v_2 - v_1$, a direct computation implies that
\bes
\mathrm{Area} = a_1 \wedge b_1 + a_2 \wedge b_2
\ees
as requested by the normalization. We conclude, since the cylinders are simple and hence do not admit any fractional Dehn twist, that 
\ba
&\fraka_i = 1\,, \quad \frakb_i = 1/2\ ({\rm for}\ i=1,2)\,, \quad R = 1\,, \quad \cL = 2\,,
\quad c^{\area}_{\Gamma_{\rm I}} = 2\,, 
\ea
where the last identity follows since there is only one prong-matching and since the bottom level is uniquely determined by the topological type, forcing the map~$\sigma$ in Definition~\ref{def:K} to have degree one.
\par
It remains to determine the factor $\cK_{\cY_{\Gamma_{\rm I}}}$. Observe that in the preceding top-level volume discussion, all the marked points, including the two upper half-edges $e_1$ and $e_2$ of the vertical edges, have been fixed. We choose a double zero to be on top, which gives $3$ choices. For the remaining two double zeros, due to the bottom symmetry, we have a unique choice to label them. Once they are labeled, each individual vertical half-edge on the bottom is distinguished and connects to $e_1$ and $e_2$, respectively. It follows that this type has $3$  components. For each component, the bottom of the graph admits an involution that swaps the two unmarked horizontal edges, while this involution maps the bottom differentials (up to simultaneously scaling) $[\omega_1, \omega_2]$ to $[-\omega_1, -\omega_2] = [\omega_1, \omega_2]$. Hence, it is a non-trivial automorphism of the concerned multi-scale differential and preserves the gothic locus.  Therefore, $| {\rm Aut}_{\cY_{\Gamma_{\rm I}}}(\Gamma_{\rm I}) | = 2$. 
\par 
In summary, the contribution of all components in Type {\rm I} to the area Siegel--Veech constant is 
$$ c_{\rm area}^{\rm int}(\Omega \mathcal{G}, \Gamma_{\rm I}) \= 3 \cdot \left(-\frac{1}{4\pi^2}\cdot 2 \cdot \frac{1}{2} \cdot \frac{-\frac{5}{4}}{\frac{13}{16}}\right)  
\=  \frac{1}{\pi^2} \cdot \frac{15}{13}\,. $$

%%%%%%%%%%%%%%%%%%%%%%%%%%%%%%%%%%%%%%
\subsubsection{Type II: The compact-type graph}
%%%%%%%%%%%%%%%%%%%%%%%%%%%%%%%%%%%%%%

From now on, we will be more brief by stating the cylinder geometry only rather than presenting the full degeneration and normalization discussion. This case has two cylinders exchanged by the involution. We choose~$\alpha_1$ as the core curve of one of them and~$\beta_1$ so that $a_1 \wedge b_1$ gives the cylinder's volume contribution to the total area. Since the cylinders have modulus $1/2$, we find $\fraka_i = 1$ and $\frakb_i = 1/2$ for $i=1,2$. Since all zeros are marked, there is no fractional Dehn twist (i.e., $R=1$) and we conclude that 
$c_{\Gamma_{\rm II}}^{\area} = 2$.
\par
Next, suppose the two top-level vertices are distinguished (i.e., their vertical edges are labeled). Then, we have $\int_{\PP\cY^\top_{\Gamma_{\rm II}}} \xi = -1/24$, due to the involution of each individual top vertex (but not swapping them).
\par
To compute the factor $\cK_{\cY_{\Gamma_{\rm II}}}$, we choose a double zero to be on the right vertex of the bottom, which gives $3$ choices of labeling. For each choice, the cardinality of the bottom level has been computed to be $c=2$ in \cite[Section 2, Case (II)]{chengothic}. It follows that this type has $6$~components. For each component, there is an involution that swaps the two top vertices and simultaneously swaps the two bottom horizontal edges. This involution maps the top differential $[\omega, -\omega]$ to $[-\omega, \omega]$ and maps the bottom differentials $[\omega_1, \omega_2]$ to $[-\omega_1, -\omega_2]$. Hence, it is an automorphism of order $2$ of the multi-scale differential and preserves the gothic locus. Therefore,  $| {\rm Aut}_{\cY_{\Gamma_{\rm II}}}(\Gamma_{\rm II}) | = 2$. 
\par 
In summary, the contribution of all components in Type {\rm II} to the area Siegel--Veech constant is 
 $$ c_{\rm area}^{\rm int}(\Omega \mathcal{G}, \Gamma_{\rm II}) \= 6\cdot \left( -\frac{1}{4\pi^2}\cdot 2 \cdot \frac{1}{2} \cdot \frac{-\frac{1}{24}}{\frac{13}{16}}\right) \= \frac{1}{\pi^2}\cdot \frac{1}{13}\,. $$

%%%%%%%%%%%%%%%%%%%%%%%%%%%%%%%%%%%%%%
\subsubsection{Type III: The asymmetric cherry}
%%%%%%%%%%%%%%%%%%%%%%%%%%%%%%%%%%%%%%

In this case, there are two pairs of cylinders exchanged by the involution (as the graph $\Gamma_{\rm III}$ shows), where one pair would be combined into a bigger cylinder if the zeros of order zero were unmarked. The flat geometric picture shows that the three combined cylinders have the same modulus, where the large one is twice as high and wide as the two smaller ones. Choosing $\alpha_1$ to be the core curve of the long cylinder
and~$\beta_1$ according to the total volume normalization, we find
\ba
&\fraka_i = 1\,, \,\,\frakb_i = 1/3 \quad ({\rm for}\ i=1,2)\,, \quad \fraka_i = 1/2\,, \,\,\frakb_i = 1/3 \quad ({\rm for}\ i=3,4)\,, \\
&R =1\,, \qquad \cL = 3\,, \qquad c_{\Gamma_{\rm III}}^{\area} = 3 \cdot \Bigl( \frac13 + \frac13 + \frac23 + \frac23\Bigr) = 6\,.
\ea
To compute the degree of $\xi$ on the top level, we break the vertical edges and label their upper ends. Then we have $\int_{\PP\cY^\top_{\Gamma_{\rm III}}} \xi = -1/4$, because the top level covers $\bP\Omega\BM_{1,1}(0)$ with degree~$6$.
\par
To compute the factor $\cK_{\cY_{\Gamma_{\rm III}}}$, we choose a double zero to be on the right vertex and choose an ordinary zero of order zero to be not on the bottom, which gives $3\cdot 3 = 9$ choices. Once the double zeros are all labeled, the vertical edges are distinguished as well. It follows that this type has $9$ components.
For each component, the bottom of the graph has an involution, which sends all bottom differentials to their opposites and which can act independently of the top-level involution. Therefore, $|{\rm Aut}_{\cY_{\Gamma_{\rm III}}}(\Gamma_{\rm III}) | = 2$. 
\par 
In summary, the contribution of all components in Type {\rm III} to the area Siegel--Veech constant is 
$$ c_{\rm area}^{\rm int}(\Omega \mathcal{G}, \Gamma_{\rm III}) \= 9 \cdot \left( -\frac{1}{4\pi^2}\cdot 6 \cdot \frac{1}{2} \cdot \frac{-\frac{1}{4}}{\frac{13}{16}} \right) \= \frac{1}{\pi^2} \cdot \frac{27}{13}\,.$$ 

%%%%%%%%%%%%%%%%%%%%%%%%%%%%%%%%%%%%%%
\subsubsection{Type IV: The V-shaped graph}
%%%%%%%%%%%%%%%%%%%%%%%%%%%%%%%%%%%%%%

The length ratios of the cylinders in this case are the same as in Type~(II) and we obtain again $c_{\Gamma_{\rm IV}}^{\area} = 2$. To compute the top-level degree of $\xi$, we break the vertical edges and choose one of them to pair with the one of the two vertices on the top left that carries the same sign with the vertex on the top right.  We thus conclude that  $\int_{\PP\cY^\top_{\Gamma_{\rm IV}}} \xi = 2 \cdot (-1/4) = - 1/2 $.
\par
To compute the factor $\cK_{\cY_{\Gamma_{\rm IV}}}$, choosing a double zero on the right vertex gives $3$ choices. Hence, this type has $3$ components. 
For each component, the two top left vertices can swap, which induces an involution automorphism of the graph, while the self involution of the top right vertex has already been taken into account in the top degree of~$\xi$. Therefore, $| {\rm Aut}_{\cY_{\Gamma_{\rm IV}}}(\Gamma_{\rm IV})|  = 2$. 
\par 
In summary, the contribution of all components in Type {\rm IV} to the area Siegel--Veech constant is 
$$ c_{\rm area}^{\rm int}(\Omega \mathcal{G}, \Gamma_{\rm IV}) \=  3 \cdot \left( -\frac{1}{4\pi^2}\cdot 2 \cdot \frac{1}{2} \cdot \frac{-\frac{1}{2}}{\frac{13}{16}}\right)  \= \frac{1}{\pi^2} \cdot \frac{6}{13}\,. $$   

%%%%%%%%%%%%%%%%%%%%%%%%%%%%%%%%%%%%%%%%%%%%%%%%%%%%%%%%%%%%

\printbibliography

\end{document}